\documentclass[12pt]{article}
\usepackage{amsmath,amssymb,amsthm}
\usepackage{color}
\usepackage[colorlinks=true, linkcolor=magenta, citecolor=cyan, urlcolor=blue]{hyperref}
\font\bb=msbm9
\font\bbi=msbm10 at 12pt
\def\N{\hbox{\bb N}}

\def\R{\hbox{\bb R}}

\def\H{\hbox{\bb H}}

\def\T{\hbox{\bb T}}
\def\Hi{\hbox{\bbi H}}

\def\Ri{\hbox{\bbi R}}
\def\Ni{\hbox{\bbi N}}

\def\Ti{\hbox{\bbi T}}
\newtheorem{theo}{Theorem}[section]
\newtheorem{pro}[theo]{Proposition}
\newtheorem{lem}[theo]{Lemma}

\newtheorem{defin}[theo]{Definition}

\newcommand{\red}{\color{red}\tt }

\title{\bf Log-Sobolev inequalities   for  	
 semi-direct product operators and applications }
\author{
Piero d'Ancona, Patrick Maheux and  Vittoria Pierfelice}

\begin{document}

\maketitle


\tableofcontents

\section{Introduction and Main Theorem}

This paper deals with results of the following type:
assume that an operator ${\cal L}_1$ satisfies
a log-Sobolev inequality with parameter
on a measure space $(X_1,\mu_1)$, and that
another operator ${\cal L}_0$ satisfies a
log-Sobolev inequality (defective  or with parameter) on
a second measure space
$(X_0,\mu_0)$. Then, we
prove a log-Sobolev inequality  (defective or with parameter)
for the 
\emph{semi-direct product operator}
  $$
  {\cal L}:= {\cal L}_0+  N_1^2(x_0){\cal L}_1
  $$
  on the space
 $X_0 \times X_1$ with respect to the measure
 $\mu=\mu_0\times\mu_1$,
for any weight function $N_1: X_0\rightarrow \Ri$ 
whose set of zeroes is negligeable.

More generally, consider $n+1$ measure spaces $(X_i,\mu_i)$,
$i=0,\dots, n$, let $X=X_{0}\times \dots \times X_{n}$ be
their product with the product measure
$\mu=\mu_{0}\otimes \dots \otimes\mu_{n}$,
and write
$\widetilde{\mu}_{i}=\mu_{0}\otimes \dots \otimes\mu_{i-1}$.
Let  
$\Gamma_i$ be a  \emph{carr\'e du champ} on $X_i$
in the sense of Bakry-Emery \cite{bakry},
with domain ${\mathcal D}(\Gamma_i)\subset L^2(\mu_i)$.
Fix $n$ real valued  weight functions
$N_{1},\dots,N_{n}$ defined on $X$ with the property that
$N_{i}$ depends only on the variables $(x_{0},\dots,x_{i-1})$,
nondegenerate in the sense that 
\begin{equation}\label{zerocond}
\widetilde{\mu}_i (\{N_i= 0\})=0, \quad i=1,\dots,n.
\end{equation}
Then we set,
for any $f:X\rightarrow \Ri$, $x=(x_{0},\dots,x_{n})\in X$,
$$
\Gamma f(x_0,x_1,\dots, x_n)=
 \Gamma_0f(\hat{x}_0)(x_0)+
\sum_{i=1}^n N^2_i(x)\Gamma_if(\hat{x}_i)(x_i)
.
$$
Here we are using the notation
$$\Gamma_if(\hat{x}_i)(x_i)$$
to mean that the carr\'e du champ $\Gamma_{i}$ acts only on the
$i$-th coordinate, the others remaining fixed.
The corresponding operator  on $X$ is  defined by
$$
{\cal L}= {\cal L}_0+  \sum_{i=1}^n N^2_i(x)
{\cal L}_i.
$$
We shall call the carr\'e du champ $\Gamma$ the
\emph{semi-direct product} of the champs
$\Gamma_{0},\dots,\Gamma_{n}$, and 
$\mathcal{L}$ the \emph{semi-direct product}
of the operators $\mathcal{L}_{i}$ 
associated with the family $N_i$.

We also recall that a carr\'e du champ $\Gamma_{0}$ is said to
satisfy the \emph{diffusion property}, or that
it is of \emph{diffusion type}, if, for all
functions $\phi$, $\chi$, $\psi$ in an algebra of functions 
${\mathcal A}$ which is dense in
its domain ${\mathcal D}(\Gamma_{0})$, the identity
\begin{equation*}
  \Gamma_{0}(\phi\chi,\psi)=
  \phi \Gamma_{0}(\chi,\psi)+\chi \Gamma_{0}(\phi,\psi)
\end{equation*}
is satisfied \cite{bakry}. 
\\

Our main goal is to prove that if each $\Gamma_{i}$
(resp.~$\mathcal{L}_{i}$) for $i=0,\dots, n-1$ is
of  diffusion type and
satisfies a log-Sobolev estimate, then
their semi-direct product also satisfies
suitable generalized log-Sobolev estimates.

We recall  the relevant  notions.
We say that $ (X_i,\mu_i, {\cal L}_i)$ satisfies 
a \emph{log-Sobolev inequality with parameter} (or super log-Sobolev inequality) if there exists a continuous, non increasing function
 $M_i : (0,+\infty)\rightarrow \Ri$ such that
\begin{equation}\label{lsip}
(H_i)\quad \quad \forall h\in {\mathcal D}({\Gamma}_i), \forall t>0,\quad 
\int_{X_i} h^2\ln \frac{h^2}{\vert\vert h\vert\vert_2^2} \,
d\mu_i
\leq
t\int_{X_i}\Gamma_i(h)\, d\mu_i
+M_i(t)\, \vert\vert h\vert\vert_2^2
\end{equation}
with $\vert\vert h\vert\vert_2^2 :=\vert\vert h\vert\vert_{L^2(\mu_i)}^2$.
Recall that the relation between the carr\'e $\Gamma_{i}$
and the associated operator $\mathcal{L}_{i}$ is expressed by
$$
\int_{X_i}\Gamma_i(h)\, d\mu_i=
\int_{X_i}{\cal L}_i h. h\, d\mu_i
$$
 on the domain  
 ${\mathcal D}({\cal L}_i)\subset {\mathcal D}({\Gamma}_i)$.

We shall say that  $ (X_0,\mu_0, {\cal L}_0)$ satisfies a defective Gross inequality (or a Gross inequality if $b=0$, see below) if  there exists $a, b>0$  
 such that
\begin{equation}\label{defgross}
(H'_0)\quad \quad \forall h\in {\mathcal D}({\cal L}_0), \quad 
\int_{X_0} h^2\ln \frac{h^2}{\vert\vert h\vert\vert_{L^{2}(\mu_{0})}^2} \,
d\mu_0
\leq
a\int_{X_0}\Gamma_0(h)\, d\mu_0
+b\, \vert\vert h\vert\vert_{L^{2}(\mu_{0})}^2.
\end{equation}
We collect some examples of such inequalities in Section \ref{examp}.

The main result of the paper is the following 

\begin{theo}\label{main}
Assume $X_{i},\mu_{i},\Gamma_{i},\mathcal{L}_{i},N_{i}$ 
as above satisfy the conditions $(H_{i})$, $i=0,\dots,n$.
Moreover, assume that 
$\Gamma_{0},\dots,\Gamma_{n-1}$ have the diffusion property.
Then we have:

 \begin{enumerate}
 \item
For any multiparameter
$t=(t_0,\dots,t_n)$ with $t_i>0$, $i=0,\dots,n$,
the following inequality holds
\begin{equation}\label{result1}
\int_{X} h^2\ln \frac{h^2}{\vert\vert h\vert\vert_{L^{2}(\mu)}^2}\,
d\mu
\leq
\int_{X}\Gamma^{(t)}(h)\, d\mu
+\int_X \left[ M_0(t_0)+ N(t, x) \right] h^2(x)\, d\mu(x)
\end{equation}
where 
$\displaystyle{\Gamma}^{(t)}= t_0{\Gamma}_0+  \sum_{i=1}^n t_i N^2_i{\Gamma}_i$
,
$\displaystyle
N(t,x)=\sum_{i=1}^n M_i(t_iN^2_i(x))$.
\item
Assume that $(H'_{0})$ holds in place of $(H_{0})$.
Then the following inequality holds
\begin{equation}\label{result2}
\int_{X} h^2\ln \frac{h^2}{\vert\vert h\vert\vert_{L^{2}(\mu)}^2}\,
d\mu
\leq
\int_{X}{\widetilde{\Gamma}}^{(t)}(h)\, d\mu
+\int_X {\widetilde{W}}(t,x)h^2(x)\, d\mu(x)
\end{equation}
where
$\displaystyle
{\widetilde{\Gamma}}^{(t)}=
a{\Gamma}_0+  \sum_{i=1}^n t_i N^2_i{\Gamma}_i $,
${\widetilde{W}}(t,x)=b+ \sum_{i=1}^n M_i(t_iN^2_i(x))$.
\item
Let $t=(s,\dots,s)\in \Ri^{n+1}$, $s>0$, and  $\Gamma^{(t)}(h)= s\Gamma(h)$ with
$\displaystyle
{\Gamma}= {\Gamma}_0+  \sum_{i=1}^n  N^2_i{\Gamma}_i$.
Then, writing
$\displaystyle
\widetilde{N}(s,x)=\sum_{i=1}^n M_i(sN^2_i(x))$,
the following inequality holds
\begin{equation}\label{result2bis}
\int_{X} h^2\ln \frac{h^2}{\vert\vert h\vert\vert_{L^{2}(\mu)}^2}\,
d\mu
\leq
s\int_{X}\Gamma(h)\, d\mu
+\int_X \left[ M_0(s)+ \widetilde{N}(s, x) \right] h^2(x)\, d\mu(x).
\end{equation}
\end{enumerate}
\end{theo}


We   notice  that the operator ${\cal L}_0$ 
plays a special role in the previous result, indeed
we can assume that  ${\cal L}_0$  satisfies 
either a super log-Sobolev inequality (Statement 1)
or a defective Gross inequality (Statement 2). 
On the other hand, for the remaining operators ${\cal L}_i$
($i=1,\dots,n$)
we are not able to replace the super log-Sobolev inequality with
a defective Gross inequality, due to the way we use
the parameters in the course of the proof.
Note that if all the  $\Gamma_{i}$'s, $i=0,\dots,n$, are of diffusion type then $\Gamma^{(t)}$ is also of diffusion type

The proof of Theorem \ref{main}
is largely inspired
by two papers: \cite{cgl} for the method of proof,
and \cite{bcl} for the proper formulation of the
assumptions. 
 Recall  that it is well-known that the logarithmic Sobolev inequality  is stable 
 under the usual direct  product of spaces 
 endowed with  probability measures (the so-called tensorization method), see for instance \cite{as}. 
 Our  Theorem \ref{main}  generalizes   this situation.
 The bulk of the proof  is contained in
Section \ref{proof} of the paper, while the following
sections are devoted to several examples and applications
of our theory.
Here, for the convenience of the reader, we
would like to illustrate the essential points of the proof in
the very special case of the Grushin operator,
which can be regarded as a semidirect product in the
sense introduced above.

Thus, assume $X_{i}= \mathbb{R}_i := \mathbb{R}$ ($i=0,1$)
with the usual Lebesgue measure which we denote by $\mu_{0}=\mu_{1}$,
so that $X=\mathbb{R}^{2}$ with the 2D Lebesgue measure
$\mu=\mu_{0}\otimes\mu_{1}$.
Consider for $\alpha>0$ the operator
$$
{\cal L} f(x_0,x_1)=
\left(\frac{\partial}{\partial{x_0}}\right)^2f(x_0,x_1)+
|x_0|^{2 \alpha}\left(\frac{\partial}{\partial{x_1}}\right)^2 f(x_0,x_1)
$$ 
with
$\frac{\partial }{\partial{x_i}}$ the usual partial derivatives.
This operator is usually called the \emph{Grushin operator} (but see
Section \ref{appli} for more general operators with the
same structure).
We note that 
${\cal L}_0=(\frac{\partial}{\partial{x_0}})^2$  and 
${\cal L}_1=(\frac{\partial}{\partial{x_1}})^2$ are standard 
1D Laplacians,  and hence
${\cal L}$ is not the usual product operator on 
$\Ri^2$. However $\mathcal{L}$ can be regarded as the
semi-direct product of $\mathcal{L}_{0}$, $\mathcal{L}_{1}$
with function $N_1(x_0)=|x_0|^{\alpha} $,
in the sense defined above.
We can write the standard super log-Sobolev inequality 
on $\mathbb{R}$ (see \cite{c}, \cite{bcl}) in the form
\begin{equation}\label{lsiR}
 \forall h\in {\mathcal D}({\cal L}_i), \forall s>0 \quad 
\int_{\R_i} h^2\ln \frac{h^2}{\vert\vert h\vert\vert_{L^{2}(\mu_{i})}^2} \,
d\mu_i
\leq
s\int_{\R_i}\Gamma_i(h)\, d\mu_i
+M_i(s) \vert\vert h\vert\vert_2^2
\end{equation}
with $M_{i}(s)= -\frac{1}{2} \ln(e^2\pi s)$,
$
\Gamma_i(h)=
\vert\frac{\partial h}{\partial{x_i}}\vert^2
$
and
$\vert\vert h\vert\vert_2^2=\int_{\R_i} h^2(x_i)\,d\mu_i(x_i)$;
this  is equivalent to  the classical Gross inequality with Gaussian measure
(\cite{c},\cite{bcl} and Section \ref{examp} below).
In particular, we see that
our assumptions $(H_0)$ and $(H_1)$ are satisfied.  

We would like to prove \eqref{result1} in this special
situation. To this end, let
$f\in C_0^1(\Ri^2)$, let $t_0,t_1>0$,
fix $x_0\in \Ri_0$ and  apply \eqref{lsiR} 
to the function $x_1\mapsto f(x_0,x_1)$ defined on $X_1$. 
We obtain, for all $s>0$, 
$$
\int_{\R_1} f^2\ln f^2\,
d\mu_1
\leq
s\int_{\R_1}\Gamma_1(f)\, d\mu_1
+M_1(s) 
\int_{\R_1} f^2(x_0,x_1)\,d\mu_1(x_1)
+ 
\vert\vert  f\vert\vert^2_{L^2(\mu_1)}\ln 
 \vert\vert f  \vert\vert^2_{L^2(\mu_1)}.
$$
For $x_0\neq 0$ we can choose $s=t_1N_1^2(x_0)$
and integrate with respect to $\mu_0$, obtaining
\begin{equation}\label{lsiR0}
\int_{\R^2} f^2\ln f^2\,
d\mu
\leq
t_1\int_{\R^2} N_1^2\,  \Gamma_1(f)\, d\mu
\,+\,
\int_{\R^2}M_1(t_1N_1^2) 
 f^2\,d\mu 
\, +  
\int_{\R_0} h^2(x_0)\ln   h^2(x_0) \,d\mu_0(x_0) 
\end{equation}
where $h(x_0)=\left( \int_{\R_1} f^2(x_0,x_1)\,d\mu_1(x_1)\right)^{1/2}$.
Now we can apply \eqref{lsiR} to the function $h$ and
this gives
\begin{equation}\label{ine1}
\int_{\R_0} h^2(x_0)\ln   h^2(x_0) \,d\mu_0(x_0) 
\leq
s\int_{\R_0}\Gamma_0(h)\, d\mu_0
+M_0(s) 
\int_{\R_0} h^2\,d\mu_0
\,+\, \vert\vert  h\vert\vert^2_{L^2(\mu_0)}\ln 
 \vert\vert h  \vert\vert^2_{L^2(\mu_0)}.
\end{equation}
The last term coincides with
$$
\vert\vert  f\vert\vert^2_{L^2(\mu)}\ln 
 \vert\vert f  \vert\vert^2_{L^2(\mu)}
 $$
 while 
 $\int_{\R_0} h^2(x_0)\,d\mu_0(x_0)\equiv\vert\vert f  \vert\vert^2_{L^2(\mu)}$.
 Now consider the term
  $$
  \int_{\R_0}\Gamma_0(h)\, d\mu_0=
   \int_{\R_0}
   \left|\frac{\partial h}{\partial{x_0}}\right|^2
   d\mu_0,
   $$
by the Cauchy-Schwarz inequality we have immediately
$$
  \left| \frac{\partial h}{\partial{x_0}}\right| 
  =\left| \frac{1}{2h}\int_{X_1} 2f
      \frac{\partial f}{\partial{x_0}}   d\mu_1\right|
      \leq
      \frac{1}{h}
      \left(
  \int_{X_1} f^2\, d\mu_1\right)^{1/2}
  \left( \int_{X_1} 
    \left|  \frac{\partial f}{\partial{x_0}} \right|^2  d\mu_1
  \right)^{1/2}
  $$
  $$
   =\left(\int_{X_1} 
    \left|  \frac{\partial f}{\partial{x_0}} \right|^2  d\mu_1
    \right)^{1/2}
$$
and hence
      $$
       \int_{\R_0}\Gamma_0(h)\, d\mu_0
      \leq
      \int_{\R^2} 
    \left|  \frac{\partial f}{\partial{x_0}} \right|^2  d\mu
    =
     \int_{\R^2}\Gamma_0(f)\, d\mu.
    $$
Coming back to (\ref{lsiR0}), 
with  $s=t_0$ in (\ref{ine1}), we get
  $$
\quad 
\int_{\R^2} f^2\ln f^2\,
d\mu
\leq
t_0\int_{\R^2}   \Gamma_0(f)\, d\mu
\,+\,
t_1\int_{\R^2} N_1^2\,  \Gamma_1(f)\, d\mu
\,+\,
M_0(t_0) 
\int_{\R^2} f^2\,d\mu 
\, +  \,
$$
$$
\int_{\R^2} M_1(t_1N_1^2) 
 f^2\,d\mu 
\,+\,
\vert\vert  f\vert\vert^2_{L^2(\mu)}\ln 
 \vert\vert f  \vert\vert^2_{L^2(\mu)}.
$$
Now writing $t=(t_0,t_1)$, 
${\Gamma}^{(t)}= t_0{\Gamma}_0+    t_1 N^2_1{\Gamma}_1$
and
$
\displaystyle
N(t,x)= M_0(t_0)+M_1(t_1N_1^2(x))$,
the last inequality can be written as
  $$
\quad 
\int_{\R^2} f^2\ln f^2\,
d\mu
\leq
 \int_{\R^2}   \Gamma^{(t)}(f)\, d\mu
\,+\,
\int_{\R^2} N(t,x)f^2(x)\,d\mu(x) 
\,+\,
\vert\vert  f\vert\vert^2_{L^2(\mu)}\ln 
 \vert\vert f  \vert\vert^2_{L^2(\mu)}
$$
which is exactly \eqref{result1} for our choice
of operators and spaces.

The plan of the paper is the following.
Section \ref{proof} is devoted to the proof of the
main result Theorem \ref{main}.
In Section \ref{examp} we collect a few examples
of super log-Sobolev inequalities arising in different
contexts, from geometry, mathematical physics
and the general theory of PDEs.
These examples are further discussed in Section \ref{appli}
where we apply our main result to obtain
explicit inequalities in several specific cases.
Finally, Section \ref{ultrabound} contains some
ultracontractive (i.e. heat kernel) bounds
which can be deduced in some
special cases from our inequalities.
The paper is concluded with
a technical Appendix where we prove 
two Hardy type inequalities in the spirit of the assumption
of Rosen's  lemma \cite[Eq.4.4.2]{d}, which are necessary for
the proof of the ultracontractive bounds of
Section  \ref{ultrabound}.

\section{Proof of the Main Theorem}\label{proof}

We begin by recalling our notations.
Let $(X_i,\mu_i)$ for $i=0,\dots,n$ be  measure spaces, $X=X_0\times \dots\times X_n$  their product and
 $\mu=\mu_0\otimes \dots\otimes\mu_n$ 
 the product measure on $X$. We also denote by
 $\widetilde{\mu}_{i}=\mu_{0}\otimes \dots \otimes\mu_{i-1}$,
 ${\overline {\mu}}_i=\mu_i\otimes \dots \otimes\mu_n$,
  ${\overline {X}}_{i}=X_i\times \dots \times X_n$, 
${\widetilde{\mu}}_{i,j}=\mu_i\otimes\dots \otimes\mu_j$ and 
${\widetilde X}_{i,j}=X_i\times \dots \times X_j$   for   $0\leq i,j\leq n$.

We denote by $\Gamma_i$ a \emph{carr\'e du champ}
defined on each $X_i$,
with domain ${\mathcal D}(\Gamma_i)\subset L^2(\mu_i)$
(see \cite{bakry}); moreover, we fix $n$ real valued 
weight functions
$N_{1},\dots,N_{n}$ on $X$ with the property that
$N_{i}$ depends only on the variables $(x_{0},\dots,x_{i-1})$,
nondegenerate in the sense that 
$\widetilde{\mu}_i (\{N_i= 0\})=0$ for all $i=1,\dots,n$.
Then we set,
for any $f:X\rightarrow \Ri$,
$$
\Gamma f(x_0,x_1,\dots, x_n)=
 \Gamma_0f(\hat{x}_0)(x_0)+
\sum_{i=1}^n N^2_i(x)\Gamma_if(\hat{x}_i)(x_i)
.
$$
Here we are using the notation
$$\Gamma_if(\hat{x}_i)(x_i)$$
to mean that the carr\'e du champ $\Gamma_{i}$ acts only on the
$i$-th coordinate, the others remaining fixed.
For brevity we shall write simply
$$
\Gamma f=
 \Gamma_0f+
\sum_{i=1}^n N^2_i\Gamma_i f
.
$$
In the same way, the corresponding operator  on $X$ is  defined by
$$
{\cal L}f= {\cal L}_0f+  \sum_{i=1}^n N^2_i
{\cal L}_if
$$
(where as before 
${\cal L}_i$ acts only on the variable $x_i$ 
while the others remain fixed).

We now state a useful lemma, 
which extends an analogous result proved in
\cite[p.99]{cgl}  for sums of squares of vector fields or 
second order differential operators without constant term.

 \begin{lem}\label{basis}
Given $f\in {\cal D}({\cal L})$,
denote by $h_{n-k}$, $k=1,\dots,n$ the functions
 \begin{equation*}
   h_{n-k}(x_0,\dots,x_{n-k})^{2}=
   \int_{ {\overline {X}}_{n-k+1}} f(x_0,x_1,\dots, x_n)^2\,d
   {\overline {\mu}}_{n-k+1}.
 \end{equation*}
 Assume that $\Gamma_{0},\dots,\Gamma_{n-1}$ are carr\'e du champ with the   diffusion property.
 Then we have
 $$
 \int_{X_{n-k}}
 \Gamma_{n-k}(h_{n-k})\, d\mu_{n-k}
 \leq
 \int_{{\overline{X}}_{n-k}}
 \Gamma_{n-k}(f)\, d{\overline{\mu}}_{n-k}.
 $$
 (where both sides are functions of the variables 
 $(x_0,x_1, \dots ,x_{n-k-1})$ only).
 \end{lem}

{\bf Proof}.
We shall write the details of the proof in the
case $n=1$ only; the general case is completely
analogous.

We recall the relation between the carr\'e du champ $\Gamma_{0}$
and the corresponding generator ${\mathcal L}_{0}$:
one introduces the bilinear form
\begin{equation*}
  \Gamma_{0}(\phi,\psi)=\frac12
  \left[{\mathcal L}_{0}(\phi \psi)-\phi {\mathcal L}_{0}(\psi)-\psi {\mathcal L}_{0}(\phi)\right]
\end{equation*}
and then the relation is given by
\begin{equation*}
  \Gamma_{0}(\phi)=\Gamma_{0}(\phi,\phi).
\end{equation*}
We must prove the inequality
\begin{equation}\label{eq:but}
  \int_{X_{0}}  \Gamma_{0}(h)d\mu_{0}\le
  \int_{X_{0}\times X_{1}}\Gamma_{0}(f)d\mu_{0}d\mu_{1}
\end{equation}
where $h$ and $f$ are related by
\begin{equation*}
  h(x_{0})=\left( \int_{X_{1}}f(x_{0},x_{1})^{2}d\mu_{1}(x_{1}) \right)^{1/2}.
\end{equation*}
We can write
\begin{equation*}
  \Gamma_{0}(h^{2})=
  \Gamma_{0}(h^{2},h^{2})=
  \Gamma_{0}(h^{2},{\textstyle\int} f^{2}d\mu_{1})=
  \int \Gamma_{0}(h^{2},f^{2})d\mu_{1}
\end{equation*}
by linearity. Then using the diffusion property
\begin{equation*}
  \Gamma_{0}(\phi\chi,\psi)=
  \phi \Gamma_{0}(\chi,\psi)+\chi \Gamma_{0}(\phi,\psi)
\end{equation*}
and by Cauchy-Schwarz  inequality in $L^2$, we have
\begin{equation*}
  \Gamma_{0}(h^{2})=
  2\int f \Gamma_{0}(h^{2},f)d\mu_{1}\le
  2 h\left(\int \Gamma_{0}(h^{2},f)^{2}d\mu_{1}\right)^{1/2}.
\end{equation*}
Using now the Cauchy-Schwarz inequality  
for   the bilinear  form $\Gamma_{0}$,
we deduce
\begin{equation*}
  \le 2h\left(\int \Gamma_{0}(h^{2})\Gamma_{0}(f)d\mu_{1}\right)^{1/2}=
  2h\Gamma_{0}(h^{2})^{1/2}\left(\int\Gamma_{0}(f)d\mu_{1}\right)^{1/2}.
\end{equation*}
Thus we have proved the inequality
\begin{equation}\label{eq:ineqbut}
  \Gamma_{0}(h^{2})\le 4 h^{2}\int \Gamma_{0}(f)d\mu_{1}.
\end{equation}
Now we notice that, again by the diffusion property,
\begin{equation*}
  \Gamma_{0}(h^{2})=4h^{2}\Gamma_{0}(h)
\end{equation*}
and together with \eqref{eq:ineqbut} this implies \eqref{eq:but}.
\\

We fix $n\ge1$ and prove the main theorem by induction on 
$k=1,\dots,n$.
Let   $t_i>0$ for $i=0, \dots,n$; we must prove
the inequality 
$$
\int_{\overline{X}_{n-k}} f^2\ln  f^2 \,d{\overline{\mu}}_{n-k}
\leq
 \int_{\overline{X}_{n-k}} \sum_{i=n-k+1}^n t_i N_i^2\,\Gamma_i(f)
 \,d{\overline{\mu}}_{n-k} 
 \leqno{(R_{k})}
$$
$$
\kern10em
  +\int_{\overline{X}_{n-k}}  \sum_{i=n-k+1}^n M_i(t_i N_i^2)f^2
  \,d{\overline{\mu}}_{n-k}\,+\,
   \int_{X_{n-k}} h_{n-k}^2\ln h_{n-k}^2  \,d{\mu}_{n-k}.
$$
Notice that both sides depend only on the variables
$(x_0,x_1,\dots,x_{n-k-1})$; here and in the following,
an integral like 
$\int_{\overline{X}_{n-k}} f  \,d{\overline{\mu}_{n-k}}$
denotes integration of the function $f(x_0,x_1,\dots, x_n)$ 
with respect to the set of variables 
$(x_{n-k}, x_{n-k+1},\dots,x_n)$
in the  measure  $\,d{\overline{\mu}_{n-k}}$.

{\em Step 1}:
We start by proving $(R_1)$ i.e.
$$
\int_{X_{n-1}\times X_n} f^2\ln  f^2 \,d\mu_{n-1}d\mu_{n}
\leq
 \int_{X_{n-1}\times X_n}  t_n N_n^2\,\Gamma_n(f)
 \,d\mu_{n-1}d\mu_{n}\,+\,
 $$
 $$
  \int_{X_{n-1}\times X_n}  M_n(t_n N_n^2)f^2
  \,d\mu_{n-1}d\mu_{n}\,+\,
   \int_{X_{n-1}}h_{n-1}^2\ln h_{n-1}^2   \,d\mu_{n-1}.
$$
In order to prove this,
we apply $(H_n)$ to the function
$x_n\in X_n \rightarrow f(x_0,x_1, \dots , x_n)$. 
Recalling the definition of $h_{n-1}$ above, we obtain for any $s>0$
$$
\int_{ X_n} f^2\ln  f^2 \,d\mu_{n}
\leq
s \int_{ X_n}  \Gamma_n(f)
 \,d\mu_{n}\,+\,
  \int_{ X_n}  M_n(s)f^2
  \,d\mu_{n}\,+\,
h_{n-1}^2\ln h_{n-1}^2.
$$
We can choose now $s=t_nN^2_n(x)+\epsilon$, $\epsilon>0$
(recall that $N_n$ does not depend on $x_n$) 
and integrate w.r.to $\mu_{n-1}$; since
$M_{n}(s)$ is non increasing, by letting $\epsilon\to0$
we obtain $(R_{1})$.
\\

{\em Step 2}: 
We now assume that $(R_{k})$ is true for some
$k$, $1\le k\le n-1$, and deduce $(R_{k+1})$.
First of all, we deal with the last term of the inequality $
(R_k)$ i.e.
$$
   \int_{X_{n-k}} h_{n-k}^2\ln h_{n-k}^2  \,d{\mu}_{n-k}.
$$
If we apply assumption $(H_{n-k})$ to the function
$x_{n-k}\in X_{n-k} 
\rightarrow  h_{n-k}(x_0,x_1,\dots, x_{n-k})$,
we obtain, for any $s>0$,
$$
\int_{ X_{n-k}} h_{n-k}^2\ln  h_{n-k}^2 \,d\mu_{n-k}
\leq
s \int_{ X_{n-k}}  \Gamma_{n-k}(h_{n-k})
 \,d\mu_{n-k}\,+\,
  \int_{ X_{n-k}}  M_{n-k}(s)
  h_{n-k}^2
  \,d\mu_{n-k}
  $$
  $$
  \,+\;
h_{n-k-1}^2\ln h_{n-k-1}^2.
$$
Indeed,
$h_{n-k-1}^2=  \int_{ X_{n-k}}  h_{n-k}^2  \,d\mu_{n-k}$.
Then we choose $s= t_{n-k}N^2_{n-k}(x)+\epsilon$,
$\epsilon>0$,
which is possible since
$N_{n-k}$ does not  depend of the variables 
$(x_{n-k},\dots, x_n)$. Applying 
Lemma \ref{basis}  to  the first term 
at the right hand side
of the previous inequality, integrating the inequality  
w.r.to
$\,d\mu_{n-k-1}$ over $X_{n-k-1}$, and letting $\epsilon\to0$
as above, we obtain
\begin{equation}\label{stepint}
\int_{ {\widetilde X}_{n-k-1,n-k}} h_{n-k}^2\ln  h_{n-k}^2 \,\,d{\widetilde{\mu}}_{n-k-1,n-k}
\leq
 \int_{ {\overline X}_{n-k-1}}  
 t_{n-k}N^2_{n-k} \Gamma_{n-k}(f)
 \,d{\overline\mu}_{n-k-1}\,+\,
\end{equation}
 $$
  \kern5em
  \int_{  {\overline X}_{n-k-1}}  M_{n-k}( t_{n-k}N^2_{n-k})
f^2
 \,d{\overline\mu}_{n-k-1}
 \,+\,
  \int_{  { X}_{n-k-1}}  h_{n-k-1}^2\ln h_{n-k-1}^2
   \,d{\mu}_{n-k-1}
   .
$$ 
On the other hand, if we integrate the inequality $(R_k)$ w.r.to
$\,d\mu_{n-k-1}$ we get
\begin{equation}\label{rkint}
\int_{\overline{X}_{n-k-1}} f^2\ln  f^2 \,d{\overline{\mu}}_{n-k-1}
\leq
 \int_{\overline{X}_{n-k-1}} \sum_{i=n-k+1}^n t_i N_i^2\,\Gamma_i(f)
 \,d{\overline{\mu}}_{n-k-1}\,+\,
\end{equation} 
 $$
  \int_{\overline{X}_{n-k-1}}  
   \sum_{i=n-k+1}^n M_i(t_i N_i^2)f^2
  \,d{\overline{\mu}}_{n-k-1}\,+\,
    \int_{\widetilde{X}_{n-k-1,n-k}} 
h_{n-k}^2\ln h_{n-k}^2  \,d{\widetilde{\mu}}_{n-k-1,n-k}.
$$
Applying (\ref{stepint}) to 
estimate the last term in (\ref{rkint}) 
we finally deduce $(R_{k+1})$, and this concludes the 
induction step.

We   are now ready to prove Statements 1 and 2 of Theorem \ref{main}.

For Statement 1, we apply $(R_k)$ with $k=n$ for $n\geq 1$. 
Then  our conclusion   (\ref{result1})  differs from $(R_n)$ only by the  expression  
$$
\int_{X_0} h_{0}^2\ln h_{0}^2  \,d{\mu}_{0}.
$$
This term is treated with the assumption $(H_0)$.
Indeed, using Lemma  \ref{basis}, we can write
 $$
  \int_{X_0} h_{0}^2\ln h_{0}^2  \,d{\mu}_{0}
  \leq
  t_0\int_{X_0}\Gamma_0(h_0)\, d\mu_0
 +M_0(t_0) \int_{X_0} h_0^2\, d\mu_0
 $$
 $$
  \leq
   t_0\int_{X}\Gamma_0(f)\, d\mu
 +M_0(t_0)\int_{X} f^2\, d\mu
 $$
and Statement 1 follows.

%

Statement 2 can be proved  exactly in the same way; 
indeed, it is sufficient 
to notice that
the variable $t_0$  and $M(t_0)$ can be
considered as two constants
$a>0$ and $b\geq 0$ (also $b \in \Ri$ can be considered),
so that $(H'_0)$ can be used in place of $(H_0)$. 
Statements 3 is just a special case of Statement 1.
The proof is completed.

\section{Examples of Logarithmic  Sobolev inequality}\label{examp}
In this section, we present several 
concrete situations
where either a Gross inequality of classical or defective type,
or a logarithmic Sobolev inequality with parameter
are satisfied. Most situations are essentially known, but
in some cases we present extensions of known results which
are not completely standard.
These examples can be used as building blocks
and combined to obtain a variety of \emph{semiproduct operators}
to which our general theory applies.
Recall that on any measure space $(X,\mu)$ the \emph{entropy}
of a non negative measurable function $f$ is defined by
\begin{equation*}
  {\rm \bf Ent}_{\mu}(f):= \int_X f\ln f\, d\mu.
\end{equation*}

 \subsection{Gross type inequalities}
We begin by recalling a few important situations where a
Gross type inequality is satisfied.

\begin{enumerate}
\item
The basic example is given by $X=\mathbb{R}^{n}$
with the standard Gaussian measure
$d\mu(x)=d\gamma_n(x)=(2\pi)^{-n/2}e^{-\frac{\vert x\vert^2}{2}}$
where $dx$ is the Lebesgue measure.
The classical Gross inequality is the following 
\begin{equation}\label{simplegross}
  \|f\|_{2}=1 \quad\implies\quad
  {\rm \bf Ent}_{\gamma_n}(f^2)
  \leq
  c\int_ {\R^n} \vert\nabla f\vert^2\,d{\gamma_n},
  \qquad c=2  
\end{equation}
where 
$\vert\nabla f\vert^2=\sum_{i=1}^n\vert \frac{\partial f}{\partial x_i}\vert ^2$.
An essential feature of the inequality
is that the constant $c=2$ does not depend 
on the space dimension $n$. Note also that the inequality is
sharp, see \cite{gr1}. 
See also inequalities \eqref{eq:logggg2} and \eqref{eq:logOmega2}
in Section \ref{sec-lsip} below,
developed from the Gross inequality for an infinite measure and   the uniform measure on an open set 
$\Omega\subset \mathbb{R}^{n}$ of finite measure.

\item
An analogous result holds on the $n$-dimensional torus $\Ti^n=[-\pi,\pi]^n$.
Denoting by $\mu_n$ the normalized uniform measure on $\Ti^n$,
we have for all $f\in C^1(\Ti^n)$ (i.e.~$f\in C^{1}(\mathbb{R}^{n})$ 
and $2\pi$-periodic in each variable)
\begin{equation}\label{tore}
  \vert\vert f\vert\vert_2=1 \quad\implies\quad
{\rm \bf Ent}_{\mu_n}(f^2)
\leq
c\int_ {\T^n} \vert\nabla f\vert^2\,d\mu_n\,\qquad
c=2. 
\end{equation}
For $n=1$, the inequality is sharp, see \cite{we} (see also \cite{ey}  for a simple proof).
Since the constant $c=2$ is independent of the dimension $n$, 
the general case $n\geq 1$ is obtained by the classical argument of tensorization.
This inequality (\ref{tore}) also extends to the infinite dimensional case.
More precisely, denote by $\Ti^{\infty}$ the infinite
dimensional torus with Haar probability measure
$\mu^{\N}$ (see \cite{b} for more details and in particular
a study of several  heat kernels on $\Ti^{\infty}$). Then the above Gross
type inequality is valid for all {\em cylindrical} functions
in $C^{1}(\Ti^{\infty})$, i.e., depending on a finite number of 
coordinates.

\item
Consider for some $L>0$
an interval $X=[0,L]\subset\Ri$ endowed
with the uniform measure $d\mu(x)=\frac{dx}{L}$; no periodic conditions at 
the boundary are imposed. Then we have
$$
\vert\vert f\vert\vert_2=1
\quad\implies\quad
{\rm \bf Ent}_{\mu}(f^2)
\leq
\frac{2L^2}{\pi^2}\int_ {[0,L]} \vert\nabla f\vert^2\,d\mu
$$
(see \cite{ey}, \cite{g} for this and related results).
Compare this result with the case of periodic boundary
conditions, where we have
$$
\vert\vert f\vert\vert_2=1 \quad\implies\quad
{\rm \bf Ent}_{\mu}(f^2)
\leq
\frac{L^2}{2\pi^2}\int_ {[0,L]} \vert\nabla f\vert^2\,d\mu. $$
When $L=2\pi$ this is contained in \eqref{tore}.

\item
The case of an interval $I=(a,b) \subset\Ri$  
and a general weighted Lebesgue measure
$d\mu(x)=\frac{1}{Z}e^{-V(x)}\,dx$, where
$Z= \int_I  e^{-V(x)}\,dx$,
is studied in \cite[Section 7]{l}.
Then the following results are valid:
\begin{enumerate}
\item
When  $I=(0,2\pi)$ and $V(x)=-2{\gamma}\ln \sin \left(\frac{x}{2}\right)$ 
with $\gamma > 1/2$,
$$
\vert\vert f\vert\vert_2=1 \quad\implies\quad
{\rm \bf Ent}_{\mu}(f^2)
\leq
\frac{8}{1+2\gamma}\int_ {I} \vert\nabla f\vert^2\,d\mu.
$$
\item
When  $I=(-1,1)$ and $V(x)=-2{\alpha}\ln(1-x)-2{\beta}\ln(1+x)$,
with
\begin{equation*}
  \gamma={\rm min\,}(\alpha,\beta),\qquad
  \delta={\rm max\,}(\alpha,\beta),
\end{equation*}
if $\gamma> 1/2$  then
$$
\vert\vert f\vert\vert_2=1 \quad\implies\quad
{\rm \bf Ent}_{\mu}(f^2)
\leq
\frac{\delta}{\gamma (1+2\delta) }\int_ {I} \vert\nabla f\vert^2\,d\mu.
$$
\end{enumerate}

\item
Let $\lambda>-\frac{1}{2}$ and the probability measure $d\mu_{\lambda}(x)=A_{\lambda}(1-x^2)^{\lambda-(1/2)}\,dx$ 
on $ I=[-1,1]$.
Then for any real-valued $f\in C^2([-1,1])$, we have the following  sharp log-Sobolev inequality
$$
\vert\vert f\vert\vert_2=1 \quad\implies\quad
{\rm \bf Ent}_{\mu_{\lambda}}(f^2)
\leq
\frac{2}{2\lambda+1} \int_I H_{\lambda}f(x)f(x)\, d\mu_{\lambda}(x),
$$
with the infinitesimal generator is given by $H_{\lambda}f(x)=-(1-x^2)f''(x)+({2\lambda+1} )xf(x)$, 
$x\in I$, (see   \cite[Th.1, p.268]{mw}. Here the "gradient" has the following form
$
\nabla_1f(x)=\sqrt{1-x^2}. f'(x)=\sqrt{1-x^2}. \nabla f(x)
$
and
$$
\int_I H_{\lambda}f(x)f(x)\, d\mu_{\lambda}(x)
=
\int_I \vert \nabla_1f(x)\vert^2\, d\mu_{\lambda}(x)
$$
(compare with 4.(b)).
\item
Similar inequalities can be proved for some classes
of manifolds, both compact and non compact;
here we prefer to skip this line of research and refer
the reader to
the survey ~\cite{gr2} and the 
book by F-Y.Wang \cite[Example 5.7.2]{wa1}).
\end{enumerate}

\subsection{Logarithmic  Sobolev inequalities  with parameter}\label{sec-lsip}

We now examine a few cases where a logarithmic Sobolev
inequality with parameter is known to hold. Notice that
the following results are partially new (although they can
be proved by suitable extension of the standard techniques).

\begin{enumerate}
\item
The simplest logarithmic Sobolev inequality with parameter
corresponds to
the choice $X=\Ri^n$ and $\mathcal{L}=\Delta$,
the usual (positive) Laplacian. Then for any $t>0$ we have
\begin{equation}\label{grossplat}
\int_{\R^n} f^2\,\ln \frac{f^2}{\vert\vert f\vert\vert^2_2}\,
dx
\leq
t\int_{\R^n} \Delta f.f\, dx+ M(t)\vert\vert f\vert\vert^2_2
\end{equation}
where $M(t)=-\frac{n}{2}\ln (\pi e^2 t)$ and $dx$ is the
Lebesgue measure. Notice that we impose no constraint on
$\|f\|_{2}$. The inequality is sharp and the extremal
functions are Gaussian functions of the form
$p_t(x)=(2\pi t)^{-n/2}  
\exp(-\frac{\vert x-a\vert^2}{2t})$, $a\in \mathbb{R}^{n}$
(see \cite{c}, \cite{bcl}).
We shall refer to \eqref{grossplat} as a
``flat'' Gross inequality, to emphasize that the supporting
space $\mathbb{R}^{n}$ is viewed as a manifold with
zero curvature, as opposed to more general inequalities on
more general, nontrivial manifolds. Important consequences
can be drawn from \eqref{grossplat}, in particulary it can
be extended to more general measures of gaussian type
(see \cite{bcl} for some examples). 
As an application of our theory, we shall extend
\eqref{grossplat} to operators of Grushin type.

We recall that inequality \eqref{grossplat} is actually
equivalent to the Gross inequality \eqref{simplegross},
as it can be proved via a simple argument. For instance,
it is possible to deduce \eqref{grossplat} from Gross in
two steps. First of all, we apply
\eqref{simplegross} to a function $g$ with weighted $L^{2}$
norm $\|g\|_{2}=1$,
and defining $f$ via 
$f^2=(2\pi)^{-n/2}e^{-\frac{\vert x\vert^2}{2}}g^2$ 
we obtain
$$
\int_{\R^n} f^2\ln f^2\,dx\leq
2\int_{\R^n} \vert  \nabla f\vert^ 2\,dx
-n-\frac{n}{2}\ln (2\pi).
$$
Notice also that the unweighted $L^{2}$ norm of
$f$ is $\|f\|_{2}=1$.
For the second step, we perform a dilation
$h_{\lambda}(x)=\lambda x$ with $\lambda>0$:
since $\Delta$ satisfies $\Delta(f \circ h_{\lambda})
={\lambda}^2(\Delta f)\circ h_{\lambda}$
changing variables in the integrals we obtain
$$
\int_{\R^n} f^2\,\ln f^{2}\,
dx
\leq
2{\lambda}^2\int_{\R^n} \Delta f.f\, dx
-\frac{n}{2}\ln (2\pi e^2{\lambda}^2).
$$
It is now sufficient to set $t=2 \lambda^{2}$
and rescale the norm of $f$ to obtain, for any function $f$,
\begin{equation}\label{eq:logggg}
\int_{\R^n} f^2\,\ln \frac{f^2}{\vert\vert f\vert\vert^2_2}\,
dx
\leq
t\int_{\R^n} \Delta f.f\, dx
-\frac{n}{2}\ln ( \pi e^2 t)\vert\vert f\vert\vert^2_2
  \end{equation}
which is precisely \eqref{grossplat}. As a side remark we observe
that the log-Sobolev inequality \eqref{grossplat} is stable by  
the dilation structure imposed by the operator, in this case
the Laplacian on $\mathbb{R}^{n}$. We shall encounter
a similar situation in the case of Grushin operators 
(see Example 1 in Section \ref{appli}).
\\
Notice also that, choosing  
$t=(\pi e^2)^{-1}$ in \eqref{eq:logggg}, we deduce that the log-Sobolev inequality 
of Gross type is satisfied for an  {\em infinite} measure i.e. the Lebesgue measure on $\Ri^n$:
\begin{equation}\label{eq:logggg2}
\int_{\R^n} f^2\,\ln \frac{f^2}{\vert\vert f\vert\vert^2_2}\,
dx
\leq
\frac{1}{\pi e^2}\int_{\R^n} \Delta f.f\, dx.
\end{equation}
Here again, the inequality \eqref{eq:logggg2} implies \eqref{eq:logggg} by dilation.

\item
When $X=\Omega$ is an open set of $\Ri^n$ with Lebesgue 
finite measure $0<\vert\Omega\vert<\infty$, we easily deduce  by restriction  the
 super log-Sobolev inequality  on $\Omega$ for  smooth function  with compact support on $\Omega$  from super log-Sobolev inequality on the whole Euclidean space.
 This formulation  corresponds to the Dirichlet problem on $\Omega$ for the Laplacian. Indeed,
for any  $f\in C_0^{\infty}(\Omega)$ we have
\begin{equation}\label{eq:logOmega}
  \int_{\Omega} f^2\,\ln \frac{f^2}{\vert\vert f\vert\vert^2_2}\,
  d\mu
  \leq
  t\int_{\Omega} \Delta f.f\, d\mu
  +\ln \left( \vert \Omega\vert (\pi  e^2t)^{-\frac{n}{2}}\right)\vert\vert f\vert\vert^2_2
\end{equation}
where $\mu$ is the normalized Lebesgue measure on $\Omega$ and
$\vert\vert f\vert\vert^2_2=\vert\vert f\vert\vert^2_{L^2(\mu)}$.
To prove \eqref{eq:logOmega} it is sufficient to
take a smooth function $f$
with compact support on $\Omega$ and
apply \eqref{eq:logggg} to
$f/\sqrt{\vert\Omega\vert}$. 
\\
Moreover, as
in the first example, we deduce a log-Sobolev inequality of Gross type on $\Omega$
for the probability measure $\mu$  by setting
$t=(\pi e^2)^{-1}\vert \Omega\vert^{2/n}$ in \eqref{eq:logOmega}.
We get 
\begin{equation}\label{eq:logOmega2}
  \int_{\Omega} f^2\,\ln \frac{f^2}{\vert\vert f\vert\vert^2_2}\,
  d\mu
  \leq
  (\pi e^2)^{-1}\vert \Omega\vert^{2/n}\int_{\Omega} \Delta f. f\, d\mu.
\end{equation}

\item
An important extension of the previous inequalities can be
obtained in the case of
Lie groups of polynomial growth (up to optimality).
Let $G$ be such a group with global dimension $D$, and
consider a H\"ormander system of left-invariant vector fields
$X=(X_1, X_2,\dots,X_m)$ on $G$.
We denote by $L$ the sub-Laplacian
$L=-\sum_{i=1}^m X_i^2$ and by $d$ the local dimension 
associated with $X$. We assume that $d\leq D$. 
Then the semigroup $(T_t)$ generated by $L$ satisfies
$$
\vert\vert T_t\vert\vert_{2\rightarrow \infty}\leq c_{0}\,t^{-n/4}
$$
for any $t>0$ and any $n\in [d,D]$ (see \cite{vsc})
($\vert\vert T \vert\vert_{p\rightarrow q}$ denotes
 the norm of the operator $T: L^p \rightarrow L^q$).
By Theorem 2.2.3 in \cite{d}, 
we obtain the logarithmic  Sobolev  inequality with parameter
$$
\int_{G} f^2\,\ln \frac{f^2}{\vert\vert f\vert\vert^2_2}\,
dx
\leq
t\int_{G} {\cal L} f.  f\, dx+ M(t)\vert\vert f\vert\vert^2_2
$$
where $M(t)=2\ln(c_{0})-\frac{n}{2}\ln (t/2)$
and $dx$ is the bi-invariant
 Haar  measure. 
Notice that,  in the non compact case, the parameter appears 
in a natural way, differently from the compact case.
In this context, the log-Sobolev inequalities \eqref{eq:logggg2} and   \eqref{eq:logOmega2} hold true with appropriate constants.

\item 
A further extension concerns the case of Lie groups
of exponential growth. We shall consider here
\emph{Damek-Ricci spaces}, also known as
\emph{harmonic $NA$ groups}
(\cite{ADY,DR1});
these solvable Lie groups include all symmetric spaces of 
non compact type and rank one.
We briefly recall the definition of the spaces.
Let $\mathfrak{n}=\mathfrak{v}\oplus\mathfrak{z}$ be an Heisenberg type algebra and let $N$ be the connected and simply connected Lie group associated with  $\mathfrak{n}$. Let $S$ be the one-dimensional extension of $N$ obtained by making $A=\mathbb{R}^+$ act 
on $N$ by homogeneous dilations. We denote by $Q$ the homogeneous dimension of $N$ and by $n$ the dimension of $S$. Let $H$ denote a vector in $\mathfrak{a}$ acting on
$\mathfrak{n}$ with eigenvalues $1/2$ and (possibly) $1$; we extend the inner product 
on $\mathfrak{n}$ to the algebra $\mathfrak{s}=\mathfrak{n}\oplus\mathfrak{a}$, by requiring $\mathfrak{n}$
and $\mathfrak{a}$ to be orthogonal and $H$ to be a unit vector. We denote by $d$ the left-invariant distance on $S$ associated with the
Riemannian metric on $S$ which agrees with the inner product 
on $\mathfrak{s}$ at the identity. 
The Riemannian manifold $(S,d)$ is usually referred to as \emph{Damek-Ricci space}.

Note that $S$ is nonunimodular in general; denote by
$\lambda$ and $\rho$ the left and right Haar measures on $S$, respectively. It is well known that the spaces $(S,d,\lambda)$ 
and $(S,d,\rho)$ are of \emph{exponential growth}.
In particular, the two following Laplacians on $S$ 
have been the object of investigation\,:
\begin{itemize}
\item[(i)] The Laplace-Beltrami operator $\Delta_S$ associated 
with the Riemannian me-tric $d$.
The operator $-\Delta_S$ is left-invariant, it is essentially self-adjoint on $L^2(S,\lambda)$ and its spectrum is the half line $[Q^2/4,\infty)$.
\item[(ii)] The left-invariant Laplacian $\mathcal L=\sum_{i=0}^{n-1}{Y}_i^2$, where $Y_0,\dots,Y_{n-1}$ are
 left-invariant vector fields such that at the identity $Y_0=H$, $\{Y_1,\ldots,Y_{m_{\mathfrak{v}}}\}$ is an orthonormal basis of $\mathfrak{v}$ 
and $\{Y_{m_{\mathfrak{v}}+1},\ldots,Y_{n-1}\}$ is an orthonormal basis of $\mathfrak{z}$. The operator $-\mathcal L$ is essentially self-adjoint on $L^2(S,\rho)$ and its spectrum is $[0,\infty)$. 
\end{itemize}
   
In case (ii) the theory of heat kernels is still under
development, and we shall focus here on the case (i)
of the Laplace-Beltrami operator $-\Delta_S$ on a
Damek-Ricci space $S$. This operator was studied in
\cite{APV}; in particular, the
following pointwise estimate on the heat kernel $h_{t}$
corresponding to the heat operator $e^{t \Delta_S}$, $t >0$
was proved in Proposition 3.1 \cite{APV}:
there exists a positive constant $C$
such that, for every $t>0$ and \,for any $r\in\mathbb R^+$, we have
\begin{equation}
0< h_t (r) \leq\,
\begin{cases}
C \, t^{-n/2}\,(1+r)^{\frac{n-1}2}\,e^{-\frac Q2r}\,e^{-\frac14 \,\{Q^2 t+\frac{r^2}t\}}
&\text{if}\hspace{3mm} t\!\le\!1\!+\!r\,,\\
C\,t^{-3/2}\,(1+r)\,e^{-\frac Q2r}\,e^{-\frac14 \,\{Q^2 t+\frac{r^2}t\}      }
&\text{if}\hspace{3mm} t \!>\!1\!+\!r.
\end{cases}
\end{equation}

Here $r$ denotes the radial variable on $S$.
As a consequence  we have
$$ 
 \|e^{t\Delta_S}\|_{1 \rightarrow \infty}\le \begin{cases}
C \,t^{-n/2}\, \,e^{-\frac{Q^2 t}{4}}
&\text{if}\hspace{3mm}0 <t \!\le\!1\!\,,\\
C \, t^{-3/2}\, \,e^{-\frac{Q^2 t}{4}}      
&\text{if}\hspace{3mm} t\!>\!1\!\,,\\
\end{cases}
$$
which implies 
$$ 
 \|e^{t\Delta_S}\|_{2 \rightarrow \infty}\le  \begin{cases}
\sqrt{C}\,t^{-n/4}\, \,e^{-\frac{Q^2 t}{8}}
&\text{if}\hspace{3mm}0 <t\!\le\!1\!\,,\\
\sqrt{C} \, t^{-3/4}\, \,e^{-\frac{Q^2 t}{8}}      
&\text{if}\hspace{3mm} t\!>\!1\!\,.\\
\end{cases}
$$
Thus we can apply again Davies' result and we obtain the 
following logarithmic
Sobolev inequality with parameter
\begin{equation}
  \int_{{S}}f^{2}\ln \frac{f^{2}}{\|f\|^{2}_{2}}dx
  \le t (-\Delta_S f,f)_{L^{2}}+M(t)\|f\|^{2}_{2}
\end{equation}
with
$$  
M(t)=  \begin{cases}
\ln(C \, 2^{n/2})
-\frac{n}{2} \ln{t} -\frac{Q^2 t}{8}
&\text{if}\hspace{3mm}0 <t\!\le\!1\!\,,\\
\ln(C \, 2^{3/2}) 
{-\frac{3}{2}} \ln{t}  -\frac{Q^2 t}{8}
&\text{if}\hspace{3mm}t\!>\!1\!\,.\\
\end{cases}
$$

\item 
Our last example is taken from the theory of Schr\"odinger
operators.
Consider on ${\Ri^n}$ the \emph{electromagnetic
Schr\"odinger operator}
$$  H=-(\nabla-i A(x))^{2}+V(x).$$
Then it is possible to prove, under very general assumptions
on the potentials $V,A$, that the heat kernel $e^{tH}(x,y)$
satisfies a pointwise gaussian estimate for all times.
In order to give the precise assumptions on the
coefficients,
we need to recall the definition of Kato class and Kato norm:

\begin{defin}\label{def.kato}
The measurable function $V(x)$ on  ${\Ri^n}$, $n\geq3$, 
is said to belong to the \emph{Kato class} if
\begin{equation}\label{eq.katoclass}
    \lim_{r\downarrow 0}\sup_{x\in{\R^n}}
      \int_{|x-y|<r}\frac{|V(y)|}{|x-y|^{n-2}}dy=0.
\end{equation}
Moreover, the \emph{Kato norm} of $V(x)$ is defined as
\begin{equation}\label{eq.katonorm}
    \|V\|_{K}=\sup_{x\in {\R^n}}
    \int_{{\R^n}}\frac{|V(y)|}{|x-y|^{n-2}}dy.
\end{equation}
For $n=2$ the kernel $|x-y|^{2-n}$ is replaced by $\ln(|x-y|^{-1})$.
\end{defin}

Then our pointwise gaussian estimate is the following 

\begin{pro}\label{pro:heatk}
  Consider the self-adjoint operator $H=-(\nabla-iA(x))^{2}+V(x)$
  on $L^{2}(\mathbb{R}^{n})$, $n\ge3$.
  Assume that $A\in L^{2}_{loc}(\mathbb{R}^{n}, \mathbb{R}^{n})$, 
  moreover $V$ is real valued and the positive
  and negative parts $V_{\pm}$ of $V$ satisfy
  \begin{equation}\label{eq:Vpiu}
    V_{+}\ \text{is of Kato class},
  \end{equation}
  \begin{equation}\label{eq:Vmeno}
    \|V_{-}\|_K< c_{n}=\pi^{n/2}/\Gamma\left(\frac{n}{2}-1\right).
  \end{equation}
  Then $e^{-tH}$ is an integral operator and its heat kernel
  $p_{t}(x,y)$ satisfies the pointwise estimate
  \begin{equation}\label{eq:heatest}
    |p_{t}(x,y)|\le
    \frac{(2\pi t)^{-n/2}}{1-  \|V_{-}\|_{K}/c_{n}}
      e^{-|x-y|^{2}/(8t)}.
  \end{equation}
\end{pro}

From estimate \eqref{eq:heatest} we obtain
$$  \|e^{-tH}\|_{1 \rightarrow \infty}\le 
      \frac{(2\pi t)^{-n/2}}{1-  \|V_{-}\|_{K}/c_{n}}$$ and as a consequence
$$  \|e^{-tH}\|_{2 \rightarrow \infty}\le 
      \frac{(2\pi t)^{-n/4}}{(1- \|V_{-}\|_{K}/c_{n})^{1/2}}
$$ 
Again by Davies' result we obtain the logarithmic
Sobolev inequality with parameter
\begin{equation}\label{eq:electromaghlogsob}
  \int_{{\R^n}}f^{2}\ln \frac{f^{2}}{\|f\|^{2}_{2}}dx
  \le t (Hf,f)_{L^{2}}+M(t)\|f\|^{2}_{2}
\end{equation}
with
$$  
     M(t)=\ln\left(
    \frac{(\pi t)^{-n/2}}{1-  \|V_{-}\|_{K}/c_{n}}
    \right).
$$
We now  give a concise but complete proof of the heat kernel estimate
in Proposition \ref{pro:heatk}.

\begin{proof}
 Assume first that the magnetic part is zero: $A\equiv 0$.
 Then  the estimate \eqref{eq:heatest} under  the assumptions
 \eqref{eq:Vpiu}, \eqref{eq:Vmeno} was proved in the paper
 \cite{pda-vp} (see the second part of Proposition 5.1).
  
  Consider now the case the magnetic part is different from zero.
  We recall Simon's diamagnetic pointwise inequality
  (see e.g.~Theorem B.13.2 in \cite{simon}), which holds under 
  the assumption $A\in L^{2}_{loc}$ (and actually even weaker):
  for any test function $g(x)$,
$$    |e^{t[(\nabla-iA(x))^{2}-V]}g|\le
    e^{t (\Delta-V)}|g|.$$
  Now, if we choose a sequence 
  $g_{\epsilon}=\epsilon^{-n}g_{1}(x/\epsilon)$ of (positive)
  test functions converging to a Dirac delta,
  we apply the estimate to the test functions
  $g=g_{\epsilon}$ translated at the point $y$, and
  let $\epsilon\to0$, 
  we obtain an analogous pointwise  inequality
  for the corresponding heat kernels:
$$    
|e^{t[(\nabla-iA(x))^{2}-V]}(x,y)|\le e^{t (\Delta-V)}(x,y) 
$$  
and \eqref{eq:heatest} follows.
\end{proof}
\end{enumerate}

\section{Examples of applications of the Main Theorem}\label{appli}

In this section we list several applications of the main
result Theorem \ref{main}. They are obtained by combining
in a suitable way the Gross type or  super log-Sobolev
inequalities examined in the previous sections. It is clear
that many more examples can be constructed by this
procedure, but we decided to focus on the most interesting ones.
Notice that some of our examples involve hypoelliptic
operators, however our proofs never use this property.

\begin{enumerate}
\item
{\it Grushin type operators on $\Ri^n$}.

Theorem \ref{main} is especially well suited to study operators
of \emph{Grushin type}, which can be defined as sums of the squares of
$n$ vector fields $Y_{1},\dots,Y_{n}$ on
$\mathbb{R}^{n}$ with the following structure:
\begin{equation}\label{grushgen}
 Y_1=\frac{\partial}{\partial x_{1}}, \quad   Y_{i}=\rho_{i}(x_{1},\dots,x_{i-1})\frac{\partial}{\partial x_{i}}, \; i=2,\dots,n.
\end{equation}
These  vector fields are divergence-free, provided the 
$\rho_{j}$'s are $C^1$; note however that our theory
does not require any smoothness of the coefficients and
can be applied also in more general situations.

In the following, we consider a few relevant cases which
exemplify the main features of the theory; it is clear that
further generalizations are possible.
\\[2mm]

$\bullet$ 
{\bf Grushin operators with polynomial coefficients.}
Consider on $\mathbb{R}^{n}$ the vector fields $Y_{1},\dots,Y_{n}$
with the structure \eqref{grushgen}, and
$\rho_{2},\dots,\rho_{n}$ are functions satisfying the
nondegeneracy condition \eqref{zerocond}, where $N_{i}=\rho_{i}$.
Note that if the $\rho_{i}$ are smooth,
the family  $(Y_{1},\dots,Y_{n})$ satisfies H\"ormander's condition
and hence the operator
\begin{equation*}
  \mathcal{L}=-\sum_{i=1}^n Y_i^2
\end{equation*}
is hypoelliptic (see \cite{b-g}). 
However, even without assuming smoothness, 
by a direct application of Theorem \ref{main}
we obtain the following logarithmic Sobolev inequality: for any $t>0$,
\begin{equation}\label{lsgrus}
  \int_{\R^n} f^2\,\ln \frac{f^2}{\vert\vert f\vert\vert^2_2}\,
dx
\leq
t\int_{\R^n} {\cal L} f. f\, dx- \frac{n}{2} \ln (\pi e^2 t)\vert\vert f\vert\vert^2_2
-\int_{\R^n} 
N(x)\vert f(x)\vert^2\,dx 
\end{equation}
where 
$N(x):=\frac{1}{2}
\sum_{i=2}^n \ln \rho^2_i(x)$ and $dx$ is the Lebesgue measure on $\Ri^n$.

It is interesting to notice that if the
operator $\mathcal{L}$
has some dilation invariance, then the log-Sobolev
inequality \eqref{lsgrus} is stable, in the following sense.
Consider for $\lambda>0$ the
non isotropic dilation on $\mathbb{R}^{n}$
\begin{equation*}
  H_{\lambda}(x)=
  ({\lambda}^{a_1}x_1,{\lambda}^{a_2}x_2, \dots ,{\lambda}^{a_n}x_n),\qquad x\in \Ri^n
\end{equation*}
with $a_i>0$, and assume that there exists an index $d>0$
such that, for any ${\lambda}>0$,
$$
{\cal L}(f \circ H_{\lambda})={\lambda}^{2d}
({\cal L}f) \circ H_{\lambda}.
$$
This is equivalent to assume that
for any $i=1, \dots ,n$, ${\lambda}>0$ and $x\in \Ri^n$
$$
{\lambda}^{2a_i}{\rho}_i^2(x)
={\lambda}^{2d}{\rho}_i^2(H_{\lambda}(x)).
$$
Then if we replace $f$ by 
$f\circ H_{\lambda}$ in (\ref{lsgrus}), 
after a change of variables, we obtain that
for any $t>0$ and $\lambda>0$
$$
  \int_{\R^n} f^2\,\ln \frac{f^2}{\vert\vert f\vert\vert^2_2}\,
dx
\leq
t{\lambda}^{2d}
\int_{\R^n} {\cal L} f. f\, dx- \frac{n}{2} \ln (e^2\pi t{\lambda}^{2d})\vert\vert f\vert\vert^2_2
-\int_{\R^n} 
N\vert f \vert^2\,dx.
$$
Now it is clear that by
setting $s=t{\lambda}^{2d}$ the previous inequality reduces to
(\ref{lsgrus}), i.e.~(\ref{lsgrus}) is stable.
\\

$\bullet$
{\bf   A special Grushin operator in dimension 2}. 
The following Grushin  operator on $\Ri^2$
has been extensively studied 
$$
{\cal L}=-\left(\frac{\partial ^2}{\partial  x^2}
+  4x^2 \frac{\partial ^2}{\partial  y^2}
\right).
$$
For instance, see 
\cite{Beckner} for the Sobolev inequality which can be related to our subject.
The associated vector fields $Y_1=\frac{\partial }{\partial  x}$ and $Y_2=2x \frac{\partial }{\partial  y}$
generate a 2-step nilpotent Lie algebra since $[Y_1,Y_2]=2\frac{\partial }{\partial  y}$ and the other brackets are zero. Moreover, the family $\{X_1,X_2\}$ satisfies the H\"ormander condition hence ${\cal L}$ is hypoelliptic.
\\

Denoting the Lebesgue measure on $\Ri^{2}$  by  
$d\mu$, our
modified  log-Sobolev inequality reads as follows: for any $t>0$,
\begin{equation}\label{eq:loglog}
  \int_{\R^2} f^2\,\ln \frac{f^2}{\vert\vert f\vert\vert^2_2}\,
d\mu
\leq
t\int_{\R^2} {\cal L} f. f\, d\mu+ M(t)\vert\vert f\vert\vert^2_2 
-  \int_{\R^2} \ln \vert x \vert\, f^2\, d\mu
\end{equation}
with $M(t):=- \ln (2 \pi e^2 t)$.
Note that:
\begin{enumerate}
\item
If $ \int_{\R^2} \ln \vert x \vert  \vert f\vert^2\, d\mu=0$, the inequality \eqref{eq:loglog}
becomes similar to the log-Sobolev inequality on $\Ri^2$ for the usual Laplacian.
\item
The operator $\mathcal{L}$
arises as the induced sub-Laplacian from $X^2+Y^2$ on the quotient of the Heisenberg group of dimension $3$  by the closed non normal subgroup generated by the vector field $X$; here $X,Y,Z$ are the usual left-invariant vector fields of the Heisenberg group. This quotient can be identified 
with $\Ri^2$ and has the structure of a  nilmanifold 
\cite[p.265]{ccfi}

\item
The super log-Sobolev inequality \eqref{eq:loglog}  with fixed $t_0>0$ such that $M(t_0)=0$ 
is a Gross type inequality 
satisfied by the operator 
$A= t_0 \,{\cal L}+V$ where $V$ is the singular potential $V(x,y)=-\ln \vert x\vert$. 
In fact, we can produce many examples of this type with the basic super log-Sobolev inequality (\ref{grossplat})
in the  Euclidean space with the Lebesgue measure, see  for instance \eqref{lsgrus}.
\item
The log-Sobolev inequality \eqref{eq:loglog} is stable by the dilation $H_{\lambda}(x,y)=(\lambda x, \lambda^2 y)$ in the sense given at the beginning of the section.
\end{enumerate}

In Section \ref{ultrabound}, we shall deduce from \eqref{eq:loglog}
the following  log-Sobolev inequality:
$$
  \int_{\R^2} f^2\,\ln \frac{f^2}{\vert\vert f\vert\vert^2_2}\,
dxdy
\leq
t\int_{\R^2} {\cal L} f. f\, dxdy+ M_1(t)\vert\vert f\vert\vert^2_2 
$$
where
$M_1(t)=\ln (c_0 \, t^{-3/2})$.
This will be obtained by estimating 
the rightmost expression
of \eqref{eq:loglog} via a 
version of
Hardy's inequality valid for logarithmic weights, see Lemma \ref{rosengr1}.
The same method can be extended to operators of the form
 $$
{\cal L}=-\left( \frac{\partial ^2}{\partial  x^2}+  (x^{2})^{m} \frac{\partial ^2}{\partial  y^2} \right),\qquad
m>0. 
$$

In \cite[Section 4]{wa2},   log-Sobolev  with parameter or  Gross type inequalities 
are proved  for the quadratic form
$
\int \Gamma(f)\, d\mu_V,
$
where $\Gamma(f)=\vert \frac{\partial f}{\partial  x}\vert^2+  x^{2m} \vert \frac{\partial f}{\partial  y}\vert^2$, $m\in \Ni$,
 and $d\mu_V=e^{-V}\, dx$ with 
 some potential $V$ related to quasi-metrics.
These examples can be used as a factor  $X_i$ in our semi-direct product theory.

$\bullet$
{\bf  A very degenerate Grushin operator.}
In \cite{f-l} the authors study a very degenerate 
diffusion  on $\Ri^2$ of the form
$$
{\cal L}=-\frac{1}{2}\frac{\partial ^2}{\partial  x^2} -  \frac{1}{2}g^2(x)\frac{\partial ^2}{\partial  y^2
}
$$
where $g(x)>c^{\prime}_1\exp(-\frac{c^{\prime}_2}{\vert x\vert^{\alpha}})$ 
with $\alpha \in [0,2[$. In particular when $\alpha\in(0,2)$
they prove the following pointwise estimate of the
heat kernel $p_{t}(x,y)$ associated with $\mathcal{L}$:
for any  ${\gamma^{\prime}}>\gamma:=\frac{\alpha}{2-\alpha}$,
\begin{equation}\label{heat}
p_t(x,y)\leq \frac{c_1}{t}
\exp\left(\frac{c_2}{t^{\gamma^{\prime}}}\right), \quad 0<t<1.
\end{equation}
Directly from this estimate one obtains the following
log-Sobolev inequality valid for $0<t<1$ :
$$
\int_{\R^2} f^2\,\ln \frac{f^2}{\vert\vert f\vert\vert^2_2}\,
d\mu
\leq
t\int_{\R^2} {\cal L} f. f\, d\mu+ M(t)\vert\vert f\vert\vert^2_2
$$
with $M(t)= c_2 \, 2^{\gamma'} t^{-\gamma'}- \ln  t+ \ln(2\, c_1)$;
here $d\mu=dxdy$ is the standard Lebesgue measure.
(Notice that in \cite{f-l} the exponent $\gamma'$ is not clearly 
defined when $t$ is large).   

We can apply our results in this situation and reverse the
process, proving a log-Sobolev inequality with parameter first,
and deducing as a corollary an upper bound on the heat kernel.
Indeed, from Theorem \ref{main} we obtain directly
for any $t>0$
  $$
  \int_{\R^2} f^2\,\ln \frac{f^2}{\vert\vert f\vert\vert^2_2}\,
d\mu
\leq
t\int_{\R^2} {\cal L}  f. f\, d\mu- \ln( \pi e^2c^{\prime}_12^{-1} t)\vert\vert f\vert\vert^2_2
+c^{\prime}_2\int_{\R^2} \frac{1\;}{\vert x\vert^{\alpha}} \vert f \vert^2\,d\mu.
$$
In Section 
\ref{appendix}, we shall prove again a Hardy type inequality 
which can be applied 
to the rightmost expression of this inequality, see Lemma \ref{rosengr2}.
This will allow to improve the above bound on $M(t)$ to
$\gamma'=\gamma=\frac{\alpha}{2-\alpha}$, and also for all
positive $t>0$, provided $\alpha\in(0,1)$.
Consequently (\ref{heat}) holds true
with  $\gamma'=\gamma:=\frac{\alpha}{2-\alpha}$ and 
for any $t>0$ improving the result of \cite{f-l}  in
the case $0<\alpha<1$. See Section \ref{ultrabound} for the full details.
Notice however that our method fails when $\alpha>1$, indeed
for any smooth function $f(x,y)$ which does not vanish near the
origin one has
$
\int_{\R^2} \frac{1\;}{\vert x\vert^{\alpha}} 
\vert f(x,y)\vert^2\,d\mu=\infty$ 
and this is an obstruction for the validity of Hardy's inequality.
In connection with this, 
it might be interesting to notice that the operator $\mathcal{L}$
is hypoelliptic when $\alpha<1$, but it is not hypoelliptic
in general when $\alpha\ge1$, and counterexamples depend on
the behaviour of $g$ near the origin.
\\

\item
{\em Semiproduct of electromagnetic Schr\"odinger operators.}

On $\mathbb{R}^{m}_{x}\times \mathbb{R}^{n}_{y}$,
consider the following operator
\begin{equation}\label{eq:potential}
  \mathcal{L}= H + a^2(x) \, K,
\end{equation}
where $H$ and $K$ are  two \emph{electromagnetic}
Schr\"odinger operators of the form
$$ H=   (i \nabla_{x}-A(x))^{2}+V(x),\qquad
  K=  (i \nabla_{y}-B(y))^{2}+W(y)
$$  with suitable potentials $V$ and $W$ and $A:\mathbb{R}^{m}\to \mathbb{R}^{m}$, $B:\mathbb{R}^{n}\to \mathbb{R}^{n}$, satisfying the assumptions of Proposition {\red \ref{pro:heatk} } and $a^2(x) >0$.
We recall that as a consequence of Proposition {\red \ref{pro:heatk}}, 
we have proved the logarithmic
Sobolev inequality  with parameter
$$  \int_{{\R^m}}f^{2}\ln \frac{f^{2}}{\|f\|^{2}_{2}}dx
  \le2t_0 (Hf,f)_{L^{2}}+M_0(t_{0})\|f\|^{2}_{2}$$
where
$$  M_0(t_0)=\ln\left(
    \frac{(2\pi t_0)^{-m/4}}{(1-  \|V_{-}\|_{K}/c_{m})^{1/2}}
  \right)
$$
and 
  $$\int_{{\R^n}}f^{2}\ln \frac{f^{2}}{\|f\|^{2}_{2}}dx
  \le2t_1 (Kf,f)_{L^{2}}+M_1(t_{1})\|f\|^{2}_{2}$$
where
$$  M_1(t_1)=\ln\left(
    \frac{(2\pi t_1)^{-n/4}}{(1- \|W_{-}\|_{K}/c_{n})^{1/2}}
  \right)
$$
(see \eqref{eq:electromaghlogsob}).
Thus, applying Theorem \ref{main}, we obtain
the following logarithmic Sobolev inequality with parameter: 
for any $h(x,y) \in \mathcal{D}(\mathcal{L})$ we have
 $$
  \int h^2\,\ln \frac{h^2}{\vert\vert h\vert\vert^2_2}\,
dxdy
\leq 
t_0 \!\! \int { H} h. h\, dxdy\ +\ 
t_1 \!\! \int a^2(x){ K} h. h\, dxdy\  + \!
 \int M(t) h^2 dxdy,
$$
 where $$M(t) = M_0(t_0) + M_1(t_1 a^2(x)).$$
It is clear that this result
can be generalized to the product of a finite number of spaces 
and to more general operators. 


\item
{\em Operators on metabelian groups and Hyperbolic spaces}.

Let $\Hi^{Q+1}=\Ri\times_t \Ri^Q$ be the semi-direct product of $\Ri$ with $\Ri^Q$ defined by
the action $t_a(n):=e^{-a}n$ for $n\in \Ri^Q$ and $a\in \Ri$.  The product law is given by
$$
(a_1,n_1)(a_2,n_2)= (a_1+a_2, e^{a_2}n_1+n_2).
$$
Consider the right-invariant vector fields
\begin{equation*}
  Y_{0}=\frac{\partial}{\partial a},\qquad
  Y_{i}=e^{a}\frac{\partial}{\partial n_{i}},\qquad
  i=1,\dots Q.
\end{equation*}
We define the full Laplacian by
$$
\mathcal{L}=-\sum_{i=0}^Q Y_i^2
$$
and we consider the action of $\mathcal{L}$ on 
$L^2(G)$ with respect to 
the left-invariant measure $dx=dadn$  on $G$ (which coincides with the 
Lebesgue measure).

The vector fields $\{Y_0,Y_1,\dots,Y_Q\}$
generate a 2-step solvable Lie algebra  since $[Y_0,Y_i]=Y_i$, 
$[Y_i,Y_j]=0$, $i,j=1,\dots,Q$,
but not nilpotent because $ad^k(Y_0)(Y_i) 
\\  = Y_i$ for any $k\in\Ni$. Moreover, the family 
$\{Y_0,Y_1,\dots,Y_Q\}$  satisfies the H\"ormander  condition  so    ${\cal L}$ is hypoelliptic.

By a direct application of the main result Theorem \ref{main} we obtain:
for all $t>0$
$$
  \int_{\H^{Q+1}} f^2\,\ln \frac{f^2}{\vert\vert f\vert\vert^2_2}\,
dx
\leq
t\int_{\H^{Q+1}} {\mathcal L} f. f\, dx+ M(t)\vert\vert f\vert\vert^2_2 -Q
  \int_{\H^{Q+1}}  a \vert f(a,n)\vert^2\, dadn
$$
where $M(t):=-\frac{(Q+1)}{2}\ln (e^2\pi t)$.

See \cite{h} for other operators in the same
class to which we can apply our theory.
Note that if for all $n$ the function 
$a\mapsto a\vert f(a,n)\vert^2$ is odd  and integrable, 
then the rightmost expression   
is zero and the super log-Sobolev inequality is formally identical 
to the same inequality on $\Ri^{Q+1}$ for the usual Laplacian.  
\\

Differently from  the   cases of  the Grushin operator and 
the very degenerate  Grushin type  operator   studied above, 
a  Hardy type inequality cannot hold for the term
$$
- \int_{\R^{Q+1}}  a \vert f(a,n)\vert^2\, dadn
$$
as it is easy to prove.
Indeed, suppose that there  are $t>0$ and $R(t)\in \Ri$ such that,  for all  $f\in C_0^{\infty}(\Ri^{Q+1})$,  
$$
- \int_{\R^{Q+1}} a \vert f\vert^2\, dadn
\leq
t\int_{\R^{Q+1}}   \left(
\left| \frac{\partial f}{\partial a}\right|^2
+
e^{2a}\vert \nabla_n f \vert^2 
\right) 
\, dadn +R(t) \int_{\R^{Q+1}}    \vert f \vert^2\, dadn.
$$
Fix   $g\in C_0^{\infty}(\Ri^{Q+1})$, $g\neq 0$. Set $f(a,n)=g(a+k,n)$ with $k\in\Ni$ in the preceding inequality
and  make the change of variables 
$u=a+k$, we  get 
$$
- \int_{\R^{Q+1}}   u \vert f \vert^2\, dudn +k \int_{\R^{Q+1}}   \vert f \vert^2\, dudn
\leq
t\int_{\R^{Q+1}}   \vert \frac{\partial f}{\partial u}   \vert^2\, dudn
\,+\, 
$$
$$
e^{-2k}
t\int_{\R^{Q+1}}  
e^{2u}\vert \nabla_n f  \vert^2
\, dudn
+
R(t) \int_{\R^{Q+1}}    \vert f  \vert^2\, dudn,
$$
which cannot hold when $k$ goes to infinity, and we
get a contradiction as claimed.
\\

\item
{\em Products of Damek-Ricci spaces.}

Consider a Riemannian manifold $X$ which is the product
of two Damek-Ricci spaces $S_{j}$, $j=0,1$:
\begin{equation*}
  X= S_0 \times S_1
\end{equation*}
and an operator $\mathcal{L}$ on $X$
defined as
\begin{equation}\label{eq:damek-ricci}
  \mathcal{L}= \mathcal{ L}_0 + a^2(x_0) \, \mathcal{ L}_1, 
\end{equation}
where $\mathcal{L}_j= -\Delta_{S_j} $ is 
the Laplace-Beltrami  operator on $S_j$ and $a^2(x_0) >0$
is a function depending only on the variable $x_{0}\in S_{0}$.
Denote by $Q_{j}$ the homogeneous dimension of $S_{j}$.
In Example 4 of Section \ref{sec-lsip} we have proved the logarithmic
Sobolev inequalities with parameter, for $j=0,1$,
$$  
\int_{{S_j}}f^{2}\ln \frac{f^{2}}{\|f\|^{2}_{2}} \, dx_j
\le
t_j (\mathcal{L}_j f,f)_{L^{2}}+M_j(t_{j})\|f\|^{2}_{2}$$
$$  
M_j(t)=  \begin{cases}
\ln(  C_{j} 2^{n_j/2})
-\frac{n_j}{2} \ln{t} -\frac{Q_{j}^2 t}{8}
&\text{if}\hspace{3mm}0 <t\!\le\!1\!\,,\\
\ln(C_{j} 2^{3/2} )
{-\frac{3}{2}} \ln{t} -\frac{Q_{j}^2 t}{8}
&\text{if}\hspace{3mm}t\!>\!1.\\
\end{cases}
$$
Thus, applying Theorem \ref{main}, we have
the following logarithmic Sobolev inequality with parameters 
on $X=S_{0}\times S_{1}$: for any 
$h(x_0,x_1) \in \mathcal{D}(\mathcal{L})$,
 $$
  \int_{X} h^2\,\ln \frac{h^2}{\vert\vert h\vert\vert^2_2}\,
dx_0dx_1
\leq
t_0 \!\! \int_{X} { \mathcal{L}_0} h. h\, dx_0dx_1+
t_1 \!\! \int_{X} a^2(x_0){ \mathcal{L}_1} h. h\, dx_0dx_1
$$
$$ + \!\!\int_{X} M(t) h^2 dx_0dx_1,
$$
where $$M(t) = M_0(t_0) + M_1(t_1 a^2(x)).$$

%
%

\item
{\em Change of variables}
 
 The class of semi-direct product operators is not stable
 by change of variables. However, we can take advantage of this fact
 in order to prove super log-Sobolev inequalities for operators which do not belong to this class,
 as we shall see in the following example. Let us  start with a formal  proposition.
 
 \begin{pro} \label{prop:formal}
 Let $M_1,M_{2}$ be two differentiable manifolds, $\mu_i$  a measure on $M_i$ and ${\mathcal  L_i}$ an operator on $M_i$  for $i=1,2$.
 Let $\Phi :M_1\rightarrow M_2$ be a $C^{\infty}$ diffeomorphism. We assume that
  $$
  \Phi_*({\mathcal L_1})={\mathcal L_2},\qquad \Phi^*(\mu_1)=\mu_2,
  $$
and that $(M_1,{\mathcal L_1},\mu_1) $ satisfies a super log-Sobolev inequality of the following  form: for any $f\in C_0^{\infty}(M_1)$ and any $t>0$,
\begin{equation}\label{slsformal1}
  \int_{M_1} f^2\,\ln \frac{f^2}{\vert\vert f\vert\vert^2_2}\,
d\mu_1
\leq  t\, \int_{M_1} {\mathcal L_1}f.f\, d\mu_1 +  \int_{M_1} B_1(t,x)f^2(x)\,d\mu_1(x).
\end{equation}
Then $(M_2,{\mathcal L_2},\mu_2) $ satisfies a super log-Sobolev inequality of the same form, namely: for any $g\in C_0^{\infty}(M_2)$, and any $t>0$,
\begin{equation}\label{slsformal2}
  \int_{M_2} g^2\,\ln \frac{g^2}{\vert\vert g\vert\vert^2_2}\,
d\mu_2
\leq  t\, \int_{M_2} {\mathcal  L_2}g.g\, d\mu_2 +  \int_{M_2} B_2(t,y)g^2(y)\,d\mu_2(y)
\end{equation}
with $B_2(t,\Phi(x))=B_1(t,x)$, $ x\in M_1$.
 \end{pro}
In the above statement,
  $  \Phi_*({\mathcal L_1})={\mathcal  L_2}$  and  $\Phi^*(\mu_1)=\mu_2$ mean  respectively:  for any $g\in C_0^{\infty}(M_2)$,
 $$
 ({\mathcal  L_2} \,g) \circ  \Phi= {\mathcal L_1} (g  \circ \Phi)
 $$  and
 $$
 \int_{M_1} (g \circ \Phi)\, d\mu_1= \int_{M_2} g \, d\mu_2.
 $$

 {\bf Proof:}
 Let $g\in C_0^{\infty}(M_2)$ and set $f=g \circ \Phi$ so that $f\in C_0^{\infty}(M_1)$.
Putting  this   function $f$
 in (\ref{slsformal1}) we deduce immediately 
 (\ref{slsformal2}).
 \\
 
 
To illustrate the interest of Proposition  \ref{prop:formal}, we apply it  to  the operator
$$
{\mathcal  L_1}f=- \frac{\partial^2 f}{ \partial^2{a}}-  e^{2a}\sum_{i=1}^Q \frac{\partial^2 f}{\partial^2{n_i}}
$$
on the space  $M_1=\Hi^{Q+1}$ with the Lebesgue measure $d\mu_1=dadn$. The change of variables 
$
\Phi:\Ri^{Q+1}\rightarrow (0,+\infty)\times \Ri^Q
$
is  given by $\Phi(a,n)=(e^a,n)=(r,n)$.
Then, starting from
\begin{equation}\label{lsihyper}
\begin{split}
  \int_{\R^{Q+1}} f^2\,\ln \frac{f^2}{\vert\vert f\vert\vert^2_2}\,
dadn
\leq 
 t\, \int_{\R^{Q+1}} {\mathcal L_1}f.f\, & dadn -Q  \int_{\R^{Q+1}} a f^2(a,n)\,dadn 
    \\
&-\frac{Q+1}{2}\ln (\pi e^2 t)\vert\vert f\vert\vert^2_2,
\end{split}
\end{equation}
we get for the transformed operator
$$
{\mathcal  L_2}  =-r^2\left( \frac{\partial^2 }{ \partial^2{r}}+\sum_{i=1}^Q \frac{\partial^2 }{\partial^2{n_i}}\right)
$$
the modified super log-Sobolev inequality
\begin{equation*}
\begin{split}
  \int_{(0,\infty)\times \R^{Q}} g^2\,\ln \frac{g^2}{\vert\vert g\vert\vert^2_2}\,
&\frac{dr}{r}dn
\leq 
 t \!\!  \int_{(0,\infty)\times\R^{Q}} {\mathcal  L_2}g.g\, \frac{dr}{r}dn
    \\
  &-Q  \!\!\int_{(0,\infty)\times\R^{Q}} ( \ln r)  g^2(r,n)\,
  \frac{dr}{r}dn
    -\frac{Q+1}{2}\ln (\pi e^2 t)\vert\vert g\vert\vert^2_2
\end{split}
\end{equation*}
where $d\mu_2(r,n)=\frac{dr}{r}dn$.

Thus we see that, although
the transformed operator is not a
a semidirect product and we can not apply our theory directly,
nevertheless we are able to
prove  a super log-Sobolev inequality.
The modified super log-Sobolev inequality just proved, 
when $Q=1$, can be interpreted as a  modified super log-Sobolev inequality
for the Laplace-Beltrami operator $\Delta=-{\mathcal L}_2$  on the Poincar\'e upper half-plane ${\bf H}$ with respect to the weighted Riemannian measure 
$d\mu(r,n)=\omega(r)r^{-2}drdn$ where $r^{-2}drdn$ is  the Riemannian measure on ${\bf H}$  and the weight 
$\omega(r)=r$. 
\\

Another interesting application can be
obtained by using as change of variables the
family of diffeomorphisms given by the
left-invariant actions
$\Phi_{g}(h)=gh$. This idea produces
super log-Sobolev inequalities for a whole family of 
second order differential operators with drift parametrized 
by elements of the group.

For simplicity we focus on the case $Q=1$.
Then, with the notations
$g=(a_1,n_1)$ and $h=(a,n)$, 
we can write explicitly $\Phi_g(h)= (a_1+a,e^an_1+n)$.
Replacing $f$ by $f\,\circ\,\Phi_g$ in the super log-Sobolev 
inequality and recalling that the measure is left-invariant, we 
arrive at
\begin{equation*}
\begin{split}
  \int_{\R^{2}} f^2\,\ln \frac{f^2}{\vert\vert f\vert\vert^2_2}\,
dadn
\leq 
t\!\! \int_{\R^{2}} {\mathcal L}^{a_1,n_1}&f.f\, dadn
-
\int_{\R^{2}} a f^2(a,n)\,dadn
\\
& + 
\left( a_1  -
\ln (e^2\pi t)\right) \vert\vert f\vert\vert^2_2
\end{split}
\end{equation*}
where
$$
{\mathcal L}^{a_1,n_1}= \
Y_0^2+(n_1^2+1)e^{-2a_1}Y_1^2
+
2n_1e^{-a_1}Y_1Y_0
+
n_1e^{-a_1} Y_1
$$
and 
$Y_0= \partial_a$, $Y_1=e^a  \partial_n$
and
 $\partial_n$ is the derivative with respect to the  variable $n$.

Each operator ${\mathcal L}^{a_1,n_1}$ is a  symmetric second-order linear differential operator  with drift on $L^2(dadn)$.
Note that the super log-Sobolev inequality 
\eqref{lsihyper} is invariant by the right-action of the group, 
thus if we use  instead the right-invariant actions,
we do not obtain new inequalities. To check invariance
we notice that
the vector fields are right-invariant, the Jacobian  determinant equals to the constant $1/\Delta(g)=e^{-Qa_1}$ where  the modular function $\Delta$  satisfies  the relation
$\Delta(g)\int f^2(xg)\,dx =\int f^2(x)\, dx$;
then we use the fact that $\ln \Delta$ is linear and the relation
$-N+\ln \Delta=0$ with $N$ as in  \eqref{lsgrus}.
\end{enumerate}

\section{Application to ultracontractive bounds}\label{ultrabound}
It is well known that from the super log-Sobolev inequality one can deduce 
suitable ultracontractive bounds and estimates on the heat kernel. 
Moreover, the Gross inequality implies the concentration phenomenon 
when the reference measure is  a probability measure (but we shall not go that direction in this paper).
In this section, we focus on the first application starting from our modified log-Sobolev inequality
(\ref{result1}) in some simple cases. But we need an additional tool, namely  Hardy type inequalities.
The general statement  is the following:

\begin{theo}\label{ultrarosen}
Let $(X,\mu)$ be a manifold with a positive measure $\mu$.
Let 
$A$ be a nonnegative selfadjoint operator on $L^2(X,\mu)$ which generates a semigroup $e^{-tA}$.
 Assume that  
 \begin{enumerate}
 \item
 $C^{\infty}_0(X)\subset {\cal D}(A)$.
 \item
For any $t>0$ and any $h\in C^{\infty}_0(X)$,
we have the following modified log-Sobolev inequality
\begin{equation}\label{hyplsi}
\int_{X} h^2\ln \frac{h^2}{\vert\vert h\vert\vert_2^2}\,
d\mu
\leq
t\,(Ah,h)+ \ln \left( {c_0} \, {t^{-n/2}}\right)
\vert \vert h\vert\vert_2^2
-\int_X  N h^2\, d\mu
\end{equation}
for some $c_0, n>0$ and some function $N:X\rightarrow \Ri$.
\item
For any  $t>0$ and any $h\in C^{\infty}_0(X)$,
we have a Hardy type inequality
\begin{equation}\label{rosen}
-\int_X  N h^2 \, d\mu \leq
t\,
(Ah,h)+g(t)\vert \vert h\vert\vert_2^2
\end{equation}
for some function $g:(0,+\infty)\rightarrow \Ri$.
 \end{enumerate}
 Then 
$$
\vert\vert T_t \vert\vert_{2\rightarrow \infty}
\leq
 c_1 \, t^{-n/4}\,
e^{M(t)} 
$$
with
$\displaystyle M(t)=(2t)^{-1} 
\int_0^t g(\varepsilon)\,d\varepsilon$
and some constant $c_1>0$.
\end{theo}

Note that
what we call a Hardy type inequality  
(\ref{rosen})  is  similar to 
the  assumption   (4.4.2) in Rosen's Lemma 4.4.1 of \cite{d} 
(provided we choose  $N=\ln \Phi$).
 The purpose of such a condition is to get
a true log-Sobolev inequality with parameter which allows us to apply Corollary 2.2.8 of \cite {d}.
\\

{\bf Proof}:
From  (\ref{hyplsi}) and (\ref{rosen}), we get  for any 
 $t>0$ and any $h\in C^{\infty}_0(X)$,
$$
\int_{X} h^2\ln \frac{\vert h \vert \; }{\vert\vert h\vert\vert_2}\,
d\mu
\leq
t\, (Ah,h)+ \left(\ln (c_0^{1/2}\, {t^{-n/4}})+2^{-1} g(t)\right)
\vert \vert h\vert\vert_2^2.
$$ 
Now it is sufficient to apply
Corollary 2.2.8 of \cite {d} to conclude the proof.
\\

We give two applications of this result. 

\begin{enumerate}
\item
Let $m>0$. 
The generalized Grushin operator on $\Ri^2$
$$
{\cal L} =- \left(\frac{\partial}{\partial{x}}\right)^2 -
(x^{2})^{m}\left(\frac{\partial}{\partial{y}}\right)^2
$$ 
which is a particular  case of  Example 1 of Section \ref{appli}, satisfies
the inequality
$$
\int_{\R^2} f^2\ln \frac{f^2}{\vert\vert f\vert\vert^2} \,d\mu
\leq
t
\vert\vert \nabla_{\cal L} f\vert\vert^2_2
-\ln(\pi e^2t)\vert\vert  f\vert\vert^2_2
- m
\int_{\R^2} \ln \vert x\vert f^2(x,y)\,d\mu
$$
for any $t>0$ where $d\mu=dxdy$.
 In order to apply Theorem \ref{ultrarosen},
we shall need
the following Hardy type inequality:
\begin{lem}\label{rosengr1}
For any $t>0$,
$$
-\int_{\R^2} \ln \vert x\vert f^2(x,y)\,dxdy
\leq
t
\vert\vert \nabla_{\cal L} f\vert\vert^2_2
+
\frac{1}{2} \ln(2e t^{-1})\vert\vert  f\vert\vert^2_2.
$$
\end{lem}
Actually this lemma will be proved with 
$\vert\vert \frac{\partial f}{\partial x} \vert\vert^2_2$ instead of  $\vert\vert \nabla_{\cal L} f\vert\vert^2_2$;
the above formulation follows from the trivial inequality 
 $\vert\vert \frac{\partial f}{\partial x} \vert\vert^2_2\leq \vert\vert \nabla_{\cal L} f\vert\vert^2_2$.
We postpone the proof to the Appendix, see Section \ref{sub:proof1}.
\\

By Theorem \ref{ultrarosen}, 
the two last inequalities imply that,
for any $t>0$,
$$
\int_{\R^2} f^2\ln  \frac{\vert f\vert \;}{\vert\vert f\vert\vert_2}
\,dxdy
\leq
t
\vert\vert \nabla_{\cal L}f\vert\vert^2_2
+
\ln\left( k \, t^{-\frac{1}{2}-\frac{m}{4}}\right)
\vert\vert  f\vert\vert^2_2
$$
for a suitable $k>0$.
Then by Corollary 2.2.8 in \cite{d}, for any $t>0$,
we deduce
$$
\vert\vert T_t f\vert\vert_{2\rightarrow \infty}
\leq
 k' \, t^{-\frac{1}{2}-\frac{m}{4}}
$$
and this implies the  following uniform bound
for the   heat kernel  $h_t$ of ${\cal L}$:
for any $p=(x,y)\in \Ri^2$ and for any $t>0$,
$$
h_t(p,p)
\leq
 k'' \, t^{-1-\frac{m}{2}}.
 $$
 Such a result is known only in the case $m=1$, at least to the authors' knowledge;
note   that here $m$ is not necessarily an integer.

Moreover, using the
 expression of the heat kernel on the diagonal
 in the case $m=1$ obtained by 
 G.~Ben Arous (see \cite[Eq.(1.20)]{bena}) via a probabilistic proof,
 we can see that our result is sharp.

%

\item
We consider the following very degenerate Grushin operator (see \cite{f-l})  defined on $\Ri^2$ by 
$$
{\cal L} =-\left(\frac{\partial}{\partial{x}}\right)^2 -
\exp\left(-\frac{2\;}{\,\vert x\vert ^{\alpha}}\right)\left(\frac{\partial}{\partial{y}}\right)^2, 
$$ 
(see Example 1 of Section \ref{appli}).
By applying the inequality \eqref{lsgrus}, we get   for any $t>0$,
$$
\int_{\R^2} f^2\ln \frac{f^2}{\vert\vert f\vert\vert^2}\,dxdy
\leq
t
\vert\vert \nabla_{\cal L} f\vert\vert^2_2
-\ln(\pi e^2t)\vert\vert  f\vert\vert^2_2
+
\int_{\R^2} 
\frac{1}{\, \vert x \vert ^{\alpha}}
f^2(x,y)\,dxdy.
$$
The following Hardy type inequality holds
(proved below in Section \ref{sub:proof2}):
\begin{lem} \label{rosengr2}
For any $0<\alpha <1$,
 there exist $c>0$ such that, for any $t>0$,
$$
\int_{\R^2} 
\frac{1}{\, \vert x\vert ^{\alpha}}
f^2(x,y)\,dxdy
\leq
t
\vert\vert \nabla_{\cal L} f\vert\vert^2_2
+
 c \, t^{-b}\vert\vert  f\vert\vert^2_2
$$
with $b=\frac{\alpha}{2-{\alpha}}$ 
(and $b$ is the unique exponent satisfying this inequality above).
\end{lem}
The remark after Lemma \ref{rosengr1} applies in a similar way
here for the gradient $\nabla_{\cal L}$.
Then by Theorem \ref{ultrarosen},  the two last inequalities imply, for any $t>0$,
$$
\int_{\R^2} f^2\ln \frac{ \vert f\vert \, }{\vert\vert f\vert\vert_2} \,dxdy
\leq
t
\vert\vert \nabla_{\cal L} f\vert\vert^2_2
+
\left(
-\frac{1}{2}
\ln\left(\pi e^2t\right)
+
\frac{c}{2} \, t^{-b}
\right)
\vert\vert  f\vert\vert^2_2.
$$
Now, following Example 2.3.4 of
\cite{d},
we conclude that for any $t>0$,
$$
\vert\vert T_t \vert\vert_{2\rightarrow \infty}
\leq 
 k' \,t^{-\frac{1}{2}} \,
\exp(c' \, t^{-b})
$$
which implies the uniform bound on the heat kernel  $h_t$ of ${\cal L}$,
for any $p=(x,y)\in \Ri^2$ and any $t>0$,
$$
h_t(p,p)
\leq
 k'' \, t^{-1}
\exp(c'' \,t^{-b}).
$$
Additional results related to this operator can be found in \cite{f-l}.
 
\end{enumerate}

\section{Appendix: Hardy type Lemmas}\label{appendix}

This appendix is devoted to the proof of the lemmas stated and used
in Section \ref{ultrabound}.
\subsection{Proof of Lemma \ref{rosengr1}}\label{sub:proof1} 
Since
\begin{equation*}
  \left\|
    \frac{\partial f}{\partial x}
  \right\|_{L^{2}(\mathbb{R}^{2})}\le 
  \|\nabla_{\cal L} f\|_{L^{2}(\mathbb{R}^{2})},
\end{equation*}
we see that it is sufficient to prove the
one dimensional estimate
\begin{equation}\label{eq:basiclog}
  \int_{\mathbb{R}} (-\ln \vert x\vert)f^2(x)\,dx
  \leq
  t
  \left\|f'\right\|_{2}^{2}
  +
 \frac{1}{2} \ln (2et^{-1})
  \vert\vert  f\vert\vert^2_2,
  \; t>0.
\end{equation}
where the norms are now in $L^{2}(\mathbb{R})$.
The proof of \eqref{eq:basiclog} will descend  from the
following proposition:

\begin{pro}\label{pro:}
  For all $0<\delta\le1$
  \begin{equation}\label{eq:hardy}
    \int_{0}^{1}- \ln x\cdot|f|^{2}dx\le
    |\ln \delta|\cdot\|f\|^{2}_{L^{2}(0,1)}
    +\|f\|^{2}_{L^{2}(0,\delta)}
    +\frac{2 \delta}{e}\|f\|_{L^{2}(0,\delta)}\|f'\|_{L^{2}(0,\delta)}.
  \end{equation}
\end{pro}

\begin{proof}
  It is sufficient to prove the estimate for a smooth function $f$.
  We have the identity
  \begin{equation*}
    \int_{0}^{1}(-\ln x)|f|^{2}dx
    =\int_{0}^{1}(-\ln x)\frac{d}{dx}\int_{0}^{x}|f|^{2}
    =\int_{0}^{1}\frac1x\int_{0}^{x}|f|^{2}+
    \left.(-\ln x)\int_{0}^{x}|f|^{2}\right|_{0}^{1}
  \end{equation*}
  and we notice that the last boundary term vanishes. We estimate the
  remaining integral at the r.h.s.~as follows. The piece of the
  integral with $\delta\le x\le1$ is bounded by
  \begin{equation}\label{eq:first}
    \int_{\delta}^{1}\frac1x\int_{0}^{x}|f|^{2}d\xi dx\le 
    \|f\|^{2}_{L^{2}(0,1)}\int_{\delta}^{1}\frac1x dx=
    \|f\|^{2}_{L^{2}(0,1)}\cdot|\ln \delta|.
  \end{equation}
  
   On the other hand, an integration by parts gives:
  \begin{equation}\label{eq:second}
  \begin{split}
    \int_{0}^{\delta}\frac1x \int_{0}^{x}|f|^{2}d\xi dx=&
    \int_{0}^{\delta}\frac1x \int_{0}^{x}\xi'|f|^{2}d\xi dx=
    \int_{0}^{\delta}\frac1x 
        \left[
          -\int_{0}^{x}2\xi f f' d\xi+
          \left.\xi |f|^{2}\right|_{0}^{x}
        \right]dx
    \\
    &
    =\int_{0}^{\delta}\int_{0}^{x}
         \left(-\frac{2\xi}{x}\right)ff'd\xi dx
      +\int_{0}^{\delta}|f|^{2}dx.
  \end{split}
  \end{equation}
  Now we notice that
  \begin{equation*}%
  \begin{split}
    \int_{0}^{\delta}\int_{0}^{x}
         \left(-\frac{2\xi}{x}\right)ff'd\xi dx=
      &
    \int_{0}^{\delta}\int_{\xi}^{\delta}
         \left(-\frac{2\xi}{x}\right)ff' dx d\xi=
    \\
      &
    = - 2\int_{0}^{\delta} f f'
        \left(\xi \ln\frac{\delta}{\xi}\right)d\xi .
  \end{split}
  \end{equation*}
  The function $\xi\ln(\delta/\xi)$ is non negative on $[0,\delta]$
  and vanishes at the boundary; its maximum is at $\xi=\delta/e$ so that
  \begin{equation*}
    \left|\xi \ln\frac{\delta}{\xi}\right|\le \frac{\delta}{e}.
  \end{equation*}
  This implies
  \begin{equation*}
  \vert 
   \int_{0}^{\delta}\int_{0}^{x}
         \left(-\frac{2\xi}{x}\right) f f' d\xi dx
         \vert
    \le \frac{2 \delta}{e}
    \|f\|_{L^{2}(0,\delta)}\|f'\|_{L^{2}(0,\delta)}
  \end{equation*}
  and by \eqref{eq:second}
  \begin{equation*}
    \int_{0}^{\delta}\frac1x \int_{0}^{x}|f|^{2}d\xi dx\le 
    \frac{2 \delta}{e}
    \|f\|_{L^{2}(0,\delta)}\|f'\|_{L^{2}(0,\delta)}
    +\|f\|^{2}_{L^{2}(0,\delta)}.
  \end{equation*}
  Putting this estimate together with \eqref{eq:first} we
  obtain \eqref{eq:hardy}.
\end{proof}

Now, we are in position to prove  Lemma \ref{rosengr1}.
By changing $f(x)$ by $f(-x)$, we get from (\ref{eq:hardy})
 \begin{equation}
    \int_{-1}^{0}- \ln \vert x\vert \cdot|f|^{2}dx\le
    |\ln \delta|\cdot\|f\|^{2}_{L^{2}(-1,0)}
    +\|f\|^{2}_{L^{2}(-\delta,0)}
    +\frac{2 \delta}{e}\|f\|_{L^{2}(-\delta,0)}\|f'\|_{L^{2}(-\delta,0
)}.
  \end{equation}

By the inequality $2ab\leq \frac{1}{s} a^2+sb^2$ valid for any $s>0$, we deduce
$$
2\|f\|_{L^{2}(-\delta,0)}\|f'\|_{L^{2}(-\delta,0
)}
+
2\|f\|_{L^{2}(0,\delta)}\|f'\|_{L^{2}(0,\delta)}
\leq
 \frac{1}{s} 
 \|f'\|^{2}_{L^{2}(-\delta,\delta)}
    +s\|f\|^{2}_{L^{2}(-\delta,\delta)}.
$$
By summing up with
(\ref{eq:hardy}), 
 for any $s>0$ and $\delta\in (0,1]$, we have
$$
   \int_{\mathbb R}- \ln \vert x\vert \cdot|f|^{2}dx\le
      \int_{-1}^{1}- \ln \vert x\vert \cdot|f|^{2}dx\le
      \frac{ \delta}{e s} \|f'\|^{2}_2
      +
   \left( |\ln \delta|+ 
\frac{ s\delta}{e}+1\right)
\|f\|^{2}_2.
$$
We have obtained an inequality of the form
$$
   \int_{\mathbb R}- \ln \vert x\vert \cdot|f|^{2}dx\le
c_1 \|f'\|^{2}_2
      +
  c_2
\|f\|^{2}_2
$$
with $c_i>0$. Here, we use a dilation argument by applying this inequality to
the rescaled function
\begin{equation*}
  f(x)=g(x \sqrt{t/c_1}\, ), \; t>0,
\end{equation*}
and we obtain
\begin{equation*}
    \int_{\mathbb R}- \ln \vert x\vert \cdot|f|^{2}dx\le
t\|f'\|^{2}_2
      +
\ln \left(\frac{
\sqrt{c_1}e^{ c_2}}{\sqrt{t}}\right)
\|f\|^{2}_2.
\end{equation*}
Taking  $c_1=  \frac{ \delta}{e s}$ and $c_2= |\ln \delta|+ 
\frac{ s \delta}{e}+1$, 
\begin{equation*}
c(s\delta):=\sqrt{c_1}e^{ c_2}
=
\sqrt{\frac{e}{s\delta} } e^{\frac{s\delta }{e}}.
\end{equation*}
We minimize $c(s\delta)$ over $s\delta>0$ and get 
$\inf_{s\delta>0} c(s\delta)=\inf_{u>0} \sqrt{\frac{1}{u}}e^u=\sqrt{2e}$,
which implies \eqref{eq:basiclog} as claimed.

\subsection{Proof of Lemma  \ref{rosengr2}}\label{sub:proof2} 
Let $0<\alpha<1,\delta>0$ and $f\in C_0^{\infty}(\Ri)$.
We write
$$
I_{\alpha}(f)=\int_0^{\infty}\frac{1}{\, \vert x\vert^{\alpha}} f^2(x)\,dx=J_{\alpha}+K_{\alpha}
$$
with $J_{\alpha}=\int_0^{\delta}\frac{1}{\vert x\vert^{\alpha}} f^2(x)\,dx$
and 
$K_{\alpha}=\int_{\delta}^{\infty}\frac{1}{\vert x\vert^{\alpha}} f^2(x)\,dx$.
Obviously,
$K_{\alpha} \leq {\delta}^{-\alpha}\,\vert\vert f\vert\vert_2^2$.
By integration by parts,
$$
J_{\alpha}=\left[ \frac{x^{1-\alpha}}{1-\alpha}f^2(x)\right]_0^{\delta}
-
\frac{2}{1-\alpha}\int_0^{\delta}{ x^{1-\alpha}}\, f\,f'\,dx
$$
$$
\leq
\frac{\delta^{1-\alpha}}{1-\alpha}f^2(\delta)+\frac{1}{1-\alpha} {\delta^{1-\alpha}}
\left( \frac{1}{\delta}\vert\vert f\vert\vert^2_2+{\delta}\vert\vert f'\vert\vert^2_2\right).
$$
This last inequality comes from:
$$
2\vert f f'\vert \leq  \frac{1}{\delta} f^2+{\delta}(f')^2.
$$
We now prove
\begin{equation}\label{holder}
\vert f(\delta)\vert\leq
\frac{1}{\sqrt{\delta}}\vert\vert f\vert\vert_2
+
\sqrt{\delta} \vert\vert f'\vert\vert_2.
\end{equation}
Let $x_0\in [0,\delta]$ such that
$\vert f(x_0)\vert =\inf_{ [0,\delta]}\vert f(x)\vert $. Then
$$
\vert f(\delta)\vert \leq \vert f(\delta)-f(x_0)\vert+\vert f(x_0)\vert
\leq
\int_0^{\delta}\vert f'\vert+ \frac{1}{\sqrt{\delta}} \vert\vert f\vert\vert_2
\leq
 {\sqrt{\delta}} \vert\vert f' \vert\vert_2
+
\frac{1}{\sqrt{\delta}} \vert\vert f\vert\vert_2
$$
by H\"older inequality. We deduce
$$
\vert f(\delta)\vert^2 \leq 
2{\delta} \vert\vert f' \vert\vert_2^2
+
\frac{2}{\delta}\vert\vert f\vert\vert_2^2.
$$
Therefore,
$$
J_{\alpha}\leq
\frac{3}{1-\alpha}\left( {\delta}^{2-\alpha}\vert\vert f' \vert\vert_2^2
+
{\delta}^{-\alpha}\vert\vert f \vert\vert_2^2
\right).
$$
From this bound and the bound on $K_{\alpha}$, we get
\begin{equation}\label{ialpha}
I_{\alpha}(f)\leq
\frac{3}{1-\alpha}\,{\delta}^{2-\alpha}\vert\vert f' \vert\vert_2^2
+\frac{4-\alpha}{1-\alpha}\,
{\delta}^{-\alpha}\vert\vert f \vert\vert_2^2.
\end{equation}
This inequality is stable by dilation.
Indeed, changing $f(x)$ by $f_{\lambda}(x)=f({\lambda}x)$, we obtain
$$
I_{\alpha}(f)\leq
\frac{3}{1-\alpha}\,({\delta {\lambda}})^{2-\alpha}\vert\vert f' \vert\vert_2^2
+\frac{4-\alpha}{1-\alpha}\,
({\delta {\lambda}})^{-\alpha}\vert\vert f \vert\vert_2^2.
$$
This reduces to (\ref{ialpha}) by setting $s={\delta {\lambda}}$.
\\

We set $c_1(\alpha)=\frac{3}{1-\alpha}$ and $c_2(\alpha) =\frac{4-\alpha}{1-\alpha}$. Let $t>0$ and choose $\delta$ such that
$t=c_1(\alpha){\delta}^{2-\alpha}$.
We set $\gamma=\frac{\alpha}{2-\alpha}$. The inequality  (\ref{ialpha}) is equivalent to
$$
I_{\alpha}(f)\leq
t\vert\vert f' \vert\vert_2^2
+c_3\, t^{-\gamma}
\vert\vert f \vert\vert_2^2.
$$
with $c_3=c_2c_1^{\gamma}$. The $L^2$-norm are the norm on $L^2(\Ri^+)$.
We easily deduce the result on $\Ri$,
$$
\int_{-\infty}^{\infty}\frac{1}{\vert x\vert^{\alpha}} f^2(x)\,dx
=
\int_0^{\infty}\frac{1}{\vert x\vert^{\alpha}} f^2(x)\,dx+\int_0^{\infty}\frac{1}{\vert x\vert^{\alpha}} f^2(-x)\,dx
\leq
t\vert\vert f' \vert\vert_2^2
+c_3 \, t^{-\gamma}
\vert\vert f \vert\vert_2^2,
$$
where now, the $L^2$-norm are the norm on $L^2(\Ri)$.
To finish the proof of Lemma \ref{rosengr2}, for $g\in C_0^{\infty}(\Ri^2)$ and any $y\in \Ri$, we set $f(x)=g(x,y)$ in the inequality just above
and integrate this inequality over $\Ri$ in $y$. We obtain
$$
\int_{\R^2} \frac{1}{\vert x\vert^{\alpha}} g^2(x,y)\,dxdy
\leq
t \, \vert\vert \frac{\partial g}{\partial x} \vert\vert_2^2
+c_3\, t^{-\gamma}
\vert\vert g \vert\vert_2^2.
$$
We conclude the lemma by the fact that
$$
  \left\|
    \frac{\partial g}{\partial x}
  \right\|^2_{L^{2}(\mathbb{R}^{2})}\le
({\cal L}g,g).
$$
We take $b=\gamma$ to prove our inequality.
\\

{\bf Proof of uniqueness of $b$}.
We use a dilation argument.
 Let $b'>0$ such that, for any $t>0$,
$$
\int_{\R^2} \frac{1}{\vert x\vert^{\alpha}} g^2(x,y)\,dxdy
\leq
t ({\cal L}g,g)  
+c_3 \, t^{-b'}
\vert\vert g \vert\vert_2^2.
$$
Replace now $g$ with $g \circ H_{\lambda}$
where $H_{\lambda}(x,y)=(\lambda x, \lambda^{\beta} y)$,  $\lambda>0$, for a fixed $\beta>1$;
after a change of variables,
 we get for any $t>0$ and $\lambda>0$:
\begin{equation}\label{predilate}
\int_{\R^2} \frac{1}{\vert x\vert^{\alpha}} g^2(x,y)\,dxdy
\leq
t{\lambda}^{2-\alpha}
({\cal L}_{\lambda}g,g)  
+c_3 \, t^{-b'}{\lambda}^{-\alpha}
\vert\vert g \vert\vert_2^2
\end{equation}
with 
$$
{\cal L}_{\lambda}={\cal L}_{\lambda,\beta}:=
-\left(\frac{\partial}{\partial{x}}\right)^2 - {\lambda}^{2\beta-2}
\exp\left(-\frac{2\lambda^{\alpha}}{\,\vert x\vert ^{\alpha}}\right)\left(\frac{\partial}{\partial{y}}\right)^2.
$$ 
Let $s>0,{\lambda}>0$ and choose $t>0$ in \eqref{predilate}  such  that $s=t{\lambda}^{2-\alpha}$,
then
$$
\int_{\R^2} \frac{1}{\vert x\vert^{\alpha}} g^2(x,y)\,dxdy
\leq
s
({\cal L}_{\lambda} g,g)  
+c_3 \, s^{-b'}{\lambda}^{-b'(\alpha-2)-\alpha}
\vert\vert g \vert\vert_2^2.
$$
Assume $b'>b$ and let $\lambda $ tend to $0$ and $s$ also (in that order), we get
$$
\int_{\R^2} \frac{1}{\vert x\vert^{\alpha}} g^2(x,y)\,dxdy=0
$$
for any function $g$: contradiction.
\\

Now, let $s>0,{\lambda}>0$ and choose $t>0$ in \eqref{predilate}  such that
$s=t^{-b'}{\lambda}^{-\alpha}$. 
Then
$$
\int_{\R^2} \frac{1}{\vert x\vert^{\alpha}} g^2(x,y)\,dxdy
\leq
s^{- \frac{1}{b'}}
 {\lambda}^{-\frac{\alpha}{b'}+2-\alpha}
({\cal L}_{\lambda}g,g)  
+c_3 \, s
\vert\vert g \vert\vert_2^2.
$$
Assume $b>b'$ and let  $\lambda$ tend to $+\infty$ and $s$ tend to $0$ (in that order), we get
the same contradiction. So $b'=b$.
The proof is completed.
\\

{\it Acknowledgements.
\\

Research partially supported by the ANR project "Harmonic Analysis at its boundaries".
ANR-12-BS01-0013-01.
\\

Patrick Maheux  benefited from two sabbatical leaves: one semester of D\'el\'egation from the CNRS (2011)  and  one semester of CRCT from the University of Orl\'eans (2012), France.}



\end{document}